\documentclass[11pt,reqno]{amsart}
\usepackage{amsmath,amsfonts,amsthm,color,amssymb, mathrsfs}
\usepackage{xcolor}
\usepackage[T1]{fontenc}
\usepackage{wasysym}
\usepackage{ulem}
\usepackage{stmaryrd} 
\usepackage[left=1in, right=1in, top=1.1in,bottom=1.1in]{geometry}
\usepackage{enumitem}
\usepackage{hyperref}
\hypersetup{
	colorlinks   = true,
	citecolor    = blue,
	linkcolor=blue
}
\usepackage{cancel}
\usepackage{comment}

\allowdisplaybreaks
\usepackage{csquotes}
\usepackage{graphicx}
\usepackage{pstricks}
\usepackage{lmodern}

\newtheorem{lemma}{Lemma}
\newtheorem{theorem}{Theorem}[section]

\newtheorem{definition}{Definition}
\newtheorem{example}{Example}
\newtheorem{proposition}{Proposition}

\newtheorem{remark}{Remark}

\def\tr{\text{tr}}

\def\la{\left\langle}
\def\ra{\right\rangle}
\def\e{\varepsilon}
\def\R{\mathbb R}
\def\E{\mathbb E}

\def\cL{\mathcal L}

\newcounter{bean}
\newcommand{\benuma}{\setlength{\labelwidth}{.25in}
	
	\begin{list}
		{(\alph{bean})}{\usecounter{bean}}}
	\newcommand{\eenuma}{\end{list}}

\begin{document}	
	\title[Mean-field control problems for BDSDEs]{On mean-field control problems for backward doubly stochastic systems}

	\author[J. Song]{Jian Song}
	\address{Research Center for Mathematics and Interdisciplinary Sciences, Shandong University, Qingdao, Shandong, 266237, China;
		and School of Mathematics, Shandong University, Jinan, Shandong, 250100, China}
	\email{txjsong@sdu.edu.cn}
	
	\author[M. Wang]{Meng Wang}
	\address{School of Mathematics, Shandong University, Jinan, Shandong, 250100, China}
	\email{wangmeng22@mail.sdu.edu.cn}
	 \date{\today}
	 		\maketitle
\begin{abstract}
  This article is concerned with stochastic control problems for  backward doubly stochastic differential equations of mean-field type, where the coefficient functions  depend on the joint distribution of the state process and the control process. We obtain the stochastic maximum principle which serves as a necessary condition for an optimal control, and we also prove its sufficiency under proper conditions. As a byproduct,  we  prove the well-posedness for a type of mean-field fully coupled forward-backward doubly stochastic differential equation arising naturally from the control problem, which is of interest in its own right. Some examples are provided to illustrate the applications of our results to control problems in the types of  scalar interaction and  first order interaction.
  	\end{abstract}

		\tableofcontents
		
		\section{Introduction}

 In this paper,  we are concerned with a control problem in which the state process $\{(y_t, z_t), t\in[0,T]\}$ is governed by the following equation		\begin{equation}\label{state}
			\left\{
			\begin{aligned}
				-dy_t=&f(t,y_t,z_t,u_t,\mathcal L(y_t,z_t,u_t))dt+ g(t,y_t,z_t,u_t,\mathcal L(y_t,z_t,u_t)) d\overleftarrow B_t\\&-z_t dW_t,\,\,\,\,t\in[0,T],\\
				y_T=&\xi.
			\end{aligned}
			\right.
		\end{equation}
In the above equation, the control process $\{u_t, t\in[0,T]\}$ is a given stochastic process;  $\mathcal L(y_t, z_t, u_t)$ stands for the law of the random vector $(y_t, z_t, u_t)$; $B$ and $W$ are two mutually independent Brownian motions;  the stochastic integral with respect to $B$ is a backward It\^o integral while the one with respect to $W$ is  forward.  This equation is called a mean-field backward doubly stochastic differential equation (MF-BDSDE) due to its dependence  on two Brownian motions as well as on the joint law of state and control processes.   The cost functional of the control problem is given by
		\begin{equation}\label{cost}
			J(u)=\E\left[ \int^T_0 h(t,y_t,z_t,u_t,\mathcal L(y_t,z_t,u_t))dt+\Phi(y_0,\mathcal L(y_0))\right].
		\end{equation}
		
Our goal of this paper is to obtain the stochastic maximum principle (SMP), a necessary condition for an optimal control,   i.e., a control minimizing $J(u)$.  	Below we briefly recall some related results,  which is by no means complete in the literature.

		Stochastic control problems have gained a particular interest due to their broad  applications in economics, finance, engineering, etc. The earliest works can be retrospected to Kushner \cite{72k} and Bismut \cite{73b}.  Among others,   the theory of general backward stochastic differential equations (BSDEs)  introduced in \cite{90pp}  plays an important role in the study of  stochastic control problems.  As an extension of BSDEs,  backward doubly stochastic differential equations  (BDSDEs) were introduced by Pardoux and Peng  in  \cite{94pp}.   We refer to Yong and Zhou \cite{99yz} and Zhang \cite{zhang17} for more details on stochastic control, BSDEs, and other related topics. 
		
	Mean-field models are useful to  characterize the asymptotic behavior when the size of the system is getting very large. Mean-field stochastic differential equations (MF-SDEs), also known as  equations  of McKean-Vlasov type,  were first introduced by Kac \cite{56k} when investigating physical systems with a large number of interacting particles. The approach of studying large particle systems pioneered by Kac now is called in the literature propagation of chaos  and  we refer to  Sznitman \cite{s91} for  further reading.   In recent years, mean-field theories for BSDEs and BDSDEs were investigated by Buckdahn et al. \cite{09bdlp} and Li and Xing \cite{lx22}, respectively.

 As is well known in the literature of game theory, it is in general hard to construct Nash equilibrium explicitly if the number of players is large.  The pioneer work of  Lasry and Lions \cite{07ll} proposed a framework of approximating Nash equilibrium for stochastic games with a large number of players.  Huang et al. \cite{06hmc} dealt with large games in a similar approach. Later on, Carmona and Delarue~\cite{13cd}  provided a  probabilistic analysis for large games formulated by Lasry and Lions, in which 	they resolved the limiting  optimal control  problem by studying   a mean-field forward-backward stochastic differential equation (MF-FBSDE).  We refer to \cite{15cd,18cd} and the references  therein for more details about   mean-field games and related topics.

  Mean-field control problems  have also attracted considerable attention  accompanying with the development of mean-field game theory. At the beginning, the investigation was focused on the control problems which involve the expected values; for instance,  Buckdahn et al. \cite{11bdl} obtained the global maximum principle for mean-field SDEs  (see also \cite{11ad}).  After Lions introduced the notion of derivatives with respect to probability measures in his seminal work \cite{lions07} (see also \cite{c12, 18cd}), a more general form of mean-field interaction where the law of the solution process is involved has been studied, see e.g. \cite{19cd,16blm}.
   We also refer to \cite{20hm,14ll,22ny} and the references  therein for more development on mean-field control problems.
		
	 Motivated by the existing works, in this paper we investigate the mean-field control problem  (\ref{state})-(\ref{cost}) for MF-BDSDEs and aim to obtain SMP. We remark that Han et al. \cite{10hpw} has obtained SMP for control problems involving such  BDSDEs without mean-field terms (see also  \cite{11sz,20zzy}). In our control problem \eqref{state}-\eqref{cost}, the state process  and the cost functional  both  depend on the joint distribution of the state process and the control process. 
Note that in our setting, the dependence on the joint distribution is rather general, and   in particular,  it includes the cases of  $\varphi(t,X_t,\E[X_t],u_t)$  and $\widetilde \E\big[\varphi(t,X_t,\widetilde X_t,u_t)\big]$	which are known as the \textit{scalar interaction} and \textit{first order interaction} of mean-field type, respectively.
These two cases will be treated  in Section \ref{example}  as  examples of applying our main result.

Let us finally summarize some difficulties and innovations of this work below.

 (i) From a modeling perspective,  BDSDE is a generalization of BSDE and hence can describe more phenomena in the real world. It is worth mentioning that  this generalization  is not trivial, for instance, classical It\^o's formula can not be directly applied due to the appearance of the backward It\^o integral. We refer to  \cite{94pp} for more details.
		
(ii) The dependence of the coefficient functions on probability measures leads to a failure of the classical  calculus. We will employ the concept of  L-derivative for  functions of probability measures initiated by P. L. Lions \cite{lions07} (see also \cite{c12,18cd}). 
			
(iii) We prove the well-posedness of the  fully coupled mean-field forward backward doubly stochastic differential equations (FBDSDEs) \eqref{mf-bdsde} which naturally arise when investigating the control problem. This type of equation was first introduced by Peng and Shi \cite{03ps} and later on was further investigated  for instance in \cite{10hpw}.

		This article is organized as follows. In Section \ref{2},  some preliminaries  of
		the L-derivative  of functions of probability measures is recalled. In Section \ref{3}, we prove our main result of stochastic maximum principle as well as a verification theorem.  Section \ref{4} is devoted to the investigation of a type of  fully coupled mean-field  BDSDE, which is of interest in its own right. Finally, we provide some example    s in Section \ref{example}.

To conclude this section, we introduce some notations that will be used throughout the article. For two vectors $u, v\in \R^n$, denote by $\la u,v\ra$  the scalar product of $u$ and $v$, by $\left\vert v \right\vert=\sqrt{\la v, v\ra}$  the Euclidean norm of $v$.  For  $A,B\in\R^{n\times d }$, we denote the scalar product of $A$ and $B$ by $\la A,B\ra=\tr\{AB^{\text{T}}\}$ and the norm of the matrix $A$ by $\left\Vert A\right\Vert=\sqrt{\tr\{AA^{\text{T}}\}}$, where the superscript $\text{T}$ stands for  the transpose of vectors or matrices. We also use the notation $\partial_x=\left(\frac{\partial}{\partial x_1},\cdots,\frac{\partial}{\partial x_n}\right)^\text{T}$ for $x\in \R^n$. Then for $\Psi:\R^n\rightarrow \R$, $\partial_x \Psi=\left( \frac\partial{\partial {x_i}} \Psi\right)_{n\times 1}$ is a column vector,  and for $\Psi:\R^n\rightarrow \R^d$,  $\partial_x \Psi=\left( \frac\partial{\partial {x_i}} \Psi_j\right)_{n\times d}$ is a $n\times  d$ matrix. Henceforth, we denote by $C$ a generic constant which can be different in different lines.

 \setcounter{equation}{0}
		\section{Preliminaries on L-derivative}\label{2}

			 In this section, we  collect some  preliminaries on L-differentiability for functions of probability laws which was initiated by P. L. Lions \cite{lions07}. We refer to \cite{c12} and \cite{18cd} for more details.
			 
For $m\in \mathbb N$, let  $\mathcal P_2(\R^m)$ be the set of probability measures on $\R^m$ with finite second moment.   	Denote by  $W_2(\cdot,\cdot)$ the 2-Wasserstein distance in $\mathcal{P}_2(\mathbb{R}^m)$, i.e., 
		$$W_2(\mu_1,\mu_2)=\inf\left\{ \left(\int_{\R^{2m}}|x-y|^2\rho(dx,dy)\right)^{\frac{1}{2}}\right\},$$
		where the infimum is taken over all $\rho\in \mathcal{P}_2(\mathbb{R}^{2m})$ with $\rho(dx,\R^m)=\mu_1(dx)$ and  $\rho(\R^m,dy)=\mu_2(dy)$. Then, $(\mathcal P_2(\R^m),W_2)$ is a polish space. It's obvious from the definition that 
		\[W(\mu_1,\mu_2)\leq \left(\E\Big[\big|X-Y\big|^2\Big]\right)^{\frac 12},\]
		here $X$ and $Y$ are $\R^m$-valued random variables with the distributions $\mu_1$ and $\mu_2$, respectively.

 For a function $H:\R^q\times \mathcal P_2(\R^m)\to \R$, we call $\widetilde H: \R^q\times L^2(\Omega;\R^m)\to \R$  a lifting of $H$ if $\widetilde H(x, Y) = H(x, \cL(Y))$,  where $\cL(Y)$ means the probability law of $Y$.  
			\begin{definition}
  A function $H: \R^q \times \mathcal P_2(\R^m)$ is said to be L-differentiable at $(x_0,\mu_0)\in \R^q\times \mathcal P_2(\R^m)$ if there exists a random variable $Y_0 \in L^2(\Omega;\R^m)$ with  $\cL(Y_0)=\mu_0$, such that the lifted function  $\widetilde H$
    is Fr\'echet differentiable at $(x_0, Y_0)$, i.e., there exists a linear continuous mapping \[[D\widetilde{H}](x_0,Y_0):\R^q\times L^2(\Omega ;\mathbb{R}^m)\rightarrow \mathbb{R}\] such that
    \begin{equation}\label{def:F-D}       \widetilde{H}(x_0+\Delta x, Y_0+\Delta Y)-\widetilde{H}(x_0,Y_0)=[D\widetilde{H}](x_0,Y_0)(\Delta x,\Delta Y)+o(|\Delta x|+ \|\Delta Y\|_{L^2}).
       \end{equation}
         \end{definition}
     
      Note that $\mathcal H: = \R^q\times L^2(\Omega;\mathbb{R}^m)$ is a Hilbert space with the inner product  \[\big< (x_1, Y_1),(x_2,Y_2)\big>_{\mathcal H}=\la x_1, x_2 \ra+\E[\la Y_1,Y_2\ra].\]  By Riesz representation theorem, the Fr\'echet derivative $[D\widetilde{H}](x_0, Y_0)$ can be viewed as an element $D\widetilde H(x_0, Y_0)$ in $\mathcal H$ in the sense that for all $(x, Y)\in \mathcal H$,
         \begin{equation}\label{e:F-D}
         [D\widetilde{H}](x_0,Y_0)(x, Y)=\big< D\widetilde{H}(x_0,Y_0), (x, Y)\big>_{\mathcal H}\,.
         \end{equation}
   Indeed, there exists a measurable function $g: \R^m\to  \R^m$  depending only on $\mu_0$ such that $D\widetilde{H}(x_0, Y)=g(Y)$ a.s. for all $Y$ with $\cL(Y)=\mu_0$.  Then, we  define the {\it L-derivative  of $H$ at $(x_0, \mu_0)$ along the random variable $Y$} by  $g(Y)$, which is denoted by $\partial_\mu H(x_0,\mu_0)(Y)$. Thus,  we have a.s.
   \[\partial_\mu H(x_0, \cL(Y)) (Y) = g(Y)= D\widetilde{H}(x_0, Y). \]

          \begin{example}\label{example-h}
          	Consider the following function $H$ on $\mathcal P_2(\R^m)$,         	\begin{align*}
          		H(\mu)=\int_{\R^m}h(y)\mu(dy),
          	\end{align*}
          where $h:\R^m\to \R$ is   twice differentiable with bounded second derivatives. Clearly, the lifted function $\widetilde H(Y)=\E\big[h(Y)\big]$ with $\cL(Y)=\mu$, and $\partial_\mu H(\cL(Y))(Y)=D\widetilde H(Y)=\partial_y h(Y)$ by \eqref{def:F-D} and \eqref{e:F-D}.  	\end{example}

  Similarly, for a function $H:\R^q\times \mathcal P_2(\R^n\times \R^{n\times l}\times \R^k)\to \R$ depending on  a vector $x\in \R^q$ and a joint probability law $\mu=(\mu_y, \mu_z, \mu_u)\in \mathcal P_2(\R^n\times \R^{n\times l}\times \R^k)$, we can define partial L-differentiability.  We say that $H$ is joint L-differentiable at $(x,\mu)$ if there exists a triple of random variables $(Y,Z,U)\in L^2(\Omega;\R^n\times \R^{n\times l}\times \R^k)$ with $\mathcal L(Y,Z,U)=\mu$ such that
  the lifted function $\widetilde H(x,Y,Z,U)=H(x,\mu)$ is Fr\'echet differentiable at $(x,Y,Z,U)$. Observing $\R^q\times L^2(\Omega;\R^n\times \R^{n\times l}\times \R^k)\cong \R^q\times L^2(\Omega;\R^n)\times L^2(\Omega; \R^{n\times l})\times L^2(\Omega;\R^k)$, 
  the partial L-derivatives $\partial_{\mu_y} H,\partial_{\mu_z} H$ and $\partial_{\mu_u} H$ at $(x, \mu)$ along  $(Y, Z, U)$ can be defined via the following identity
  \begin{align*}
  D\widetilde H(x,Y,Z,U)
  =\Big(\partial_xH(x,\mu),\partial_{\mu_y}H(x,\mu),\partial_{\mu_z}H(x,\mu),\partial_{\mu_u}H(x,\mu)\Big)(Y,Z,U).
\end{align*}
We remark that $\partial_x H(x,\mu) (Y, Z, U)$ actually does not  depend on $(Y, Z, U)$.  
 
 A standard result says  that joint continuous differentiability in the two arguments is equivalent to partial differentiability in each of the two arguments and joint
continuity of the partial derivatives. Hence, the joint continuity of $\partial_x H(x,\mu)$ means the joint continuity with respect to the Euclidean distance on $\R^q$ and the
2-Wasserstein distance on $\mathcal P_2(\R^n\times \R^{n\times l}\times \R^k)$; the joint continuity of $\partial_{\mu_y}H(x,\mu)$ is understood as the
joint continuity of the mapping $(x,Y,Z,U)\mapsto \partial_{\mu_y}H(x,\cL(Y, Z, U))(Y,Z,U)$ from $\R^q\times L^2(\Omega;\R^n\times \R^{n\times l}\times \R^k)$ to $L^2(\Omega;\R^n)$.\\

 \setcounter{equation}{0}

	\section{Stochastic maximum principle}\label{3}
	In this section, we aim to derive our main result of the stochastic maximum principle. First, we fix some mathematical notations,  formulate our control problem, and recall It\^o's formula for stochastic processes involving backward It\^o's integral. Then we present the assumptions which will be used throughout the paper. The  maximum principle  will be obtained via the  classical variational method. Assuming proper convexity conditions on the Hamiltonian, we prove a verification theorem, i.e., showing that the stochastic maximum principle is also a sufficient  condition for an optimal control.
	
			 \subsection{Some preliminaries for the control problem}
	
  On a probability space  $(\Omega,{\mathcal{F}},P)$ satisfying usual conditions, let $\{B_t\}_{t\geq 0}$ and $\{W_t\}_{t\geq 0}$ be two mutually independent Brownian motions, taking values in $\R^d$ and $\R^l$ respectively. Denote by $\mathcal N$  the collection of $P$-null sets of $\mathcal F$. For each $t\in[0,T]$, denote 
	\[\mathcal F_t=\mathcal F_t^W\vee\mathcal F_{t,T}^B\, ,\]
	where $\mathcal F_t^W=\sigma\big\{W_r,0\leq r\leq t\big\}
	\vee \mathcal N$ is the augmented $\sigma$-field generated by $W$ and similarly $\mathcal F_{t,T}^B=\sigma\big\{B_r-B_t,t\leq r\leq T\big\}
	\vee \mathcal N$. We stress that $\mathcal F_t$ is neither increasing nor decreasing in $t$ and hence does not constitute a  filtration. Now let us introduce the following spaces: 
	\begin{align*}
		&L_{\mathcal{G}}^2(\mathbb{R}^n)=\Big\{\xi:\Omega\rightarrow\mathbb{R}^n; \xi\in \mathcal{G}\text{ and }\  \E\left[|\xi|^2\right]<+\infty\Big\} \text{ for any $\sigma$-field}\  \mathcal G\subset \mathcal F;\\
		&L^2_{\mathcal{F}}([s, r];\mathbb{R}^n)=\Big\{\phi:[s,r]\times\Omega\rightarrow \mathbb{R}^n; \phi_t\in \mathcal F_t \text{ for } t\in[s,r] \text{ and } \E\left[\int^r_s |\phi_t|^2dt\right]<+\infty\Big\};\\
		&S^2_{\mathcal{F}}([s, r];\mathbb{R}^n)=\Big\{\phi:[s,r]\times\Omega\rightarrow \mathbb{R}^n; \phi  \text{ is continuous a.s., } \phi_t\in \mathcal F_t \text{ for } t\in[s,r], \\
		& \hspace{9cm}\text{ and } \E\Big[\sup_{s\leq t\leq r} |\phi_t|^2\Big]<+\infty\Big\}.
	\end{align*}

The state process $(y_t,z_t)_{0\le t\le T}$ is governed by the following BDSDE
\begin{align}\label{state-e}
	\left\{
	\begin{aligned}
	-dy_t=&f(t,y_t,y_t,u_t,\mathcal L(y_t,z_t,u_t))dt+ g(t,y_t,z_t,u_t,\mathcal L(y_t,z_t,u_t)) d\overleftarrow B_t\\&-z_t dW_t,\, t\in[0,T],\\
	y_T=&\xi,
	\end{aligned}
	\right.
\end{align}
with $\xi$ a given $\mathcal F_T$-measurable random variable. We aim to  minimize the cost functional given by
	\begin{equation}\label{cost-e}
		\begin{aligned}
			J(u)=\E\left[ \int^T_0 h(t,y_t,z_t,u_t,\mathcal L(y_t,z_t,u_t))dt+\Phi(y_0,\mathcal L(y_0))\right],
		\end{aligned}
	\end{equation}
over the set $\mathcal U:=L^2_{\mathcal F}([0,T];U)$ of admissible controls, where $U$ is a closed convex subset of $\R^k$. 
 The  functions  $f$, $g$ and  $h$ are measurable mappings from $[0,T]\times \R^n\times \R^{n\times l}\times \R^k \times \mathcal P_2(\R^n\times \R^{n\times l}\times \R^k )$ to $\R^n$, $\R^{n\times d}$ and $\R$, respectively.
		
	We stress that the state process $(y, z)$  and the cost function $J(u)$ depend on the joint distribution $\cL(y_t, z_t, u_t)$ of the state and the control processes. 	
	
	To end this subsection, we recall It\^o's formula obtained in \cite[Lemma 1.3]{94pp}, which is a key ingredient in our analysis. 
	\begin{lemma}\label{lem:ito}
	Let $\alpha\in S_{\mathcal{F}}^2([0,T]; \R^n),\beta\in L_{\mathcal{F}}^2([0,T]; \R^n), \gamma \in L_{\mathcal{F}}^2([0,T]; \R^{n\times d}), \theta\in L_{\mathcal{F}}^2([0,T]; \R^{n\times l})$ be such that
	\[\alpha_t =\alpha_0 +\int_0^t \beta_s ds+\int_0^t \gamma_s d\overleftarrow{B}_s +\int_0^t \theta_s dW_s, 0\le t \le T. \]
	Then for $\phi\in C^2(\R^n)$, we have 
\begin{align}
\phi(\alpha_t)=&\phi(\alpha_0)+\int_0^t \big< \partial_x\phi(\alpha_s), \beta_s \big> ds+ \int_0^t \big< \partial_x\phi(\alpha_s), \gamma_s d\overleftarrow{B}_s\big> + \int_0^t \big< \partial_x\phi(\alpha_s), \theta_sdW_s \big>\notag \\
&\qquad\quad  -\frac12\int_0^t \text{tr}\Big[\partial^2_{xx}\phi(\alpha_s)\gamma_s \gamma_s^{\text{T}} \Big]ds +  \frac12\int_0^t \text{tr}\big[\partial^2_{xx}\phi(\alpha_s)\theta_s \theta_s^{\text{T}} \Big]ds. \label{e:ito}
\end{align}
\end{lemma}

The following product rule is a direct corollary of Lemma \ref{lem:ito}.
 
\begin{lemma}\label{lem:prod-rule}
Consider the processes $y$ and $p$ given by 
	\begin{equation*}
		\begin{cases}
			dy_t= f_t dt + g_t d\overleftarrow{B}_t +z_t dW_t,\vspace{0.2cm}\\
			dp_t=F_tdt+G_tdW_t+q_td\overleftarrow{B}_t,
		\end{cases}
	\end{equation*}
where $f, g, z, F, G, q$ all belong to $L^2_{\mathcal F}([0,T]; \R^m)$ with proper dimension $m$.  We have
\begin{equation}\label{e:prod-rule}
	d\la p_t, y_t\ra= \la dp_t ,  y_t \ra + \la p_t, dy_t\ra +\big( G_tz_t - g_tq_t\big)dt.  
\end{equation}
\end{lemma}

	\subsection{Main assumptions and the variational equation}
	
	  We assume the following conditions for our  control problem \eqref{state-e}-\eqref{cost-e}.
	\begin{enumerate}
		\item [(H1)]  The functions $f(t,0,0,0,\delta_0)$ and $g(t,0,0,0,\delta_0)$ are uniformly bounded, where $\delta_0$ is the Dirac measure at $0$. The functions $f$, $g$ and $h$ are differentiable with respect to $(y,z,u)\in \R^n\times \R^{n\times l}\times \R^k$ for each $t\in[0,T]$ and $\mu\in\mathcal P_2(\R^{n}\times\R^{n\times l}\times \R^{k})$. Moreover,  for  $\rho=y,z,u$,  the partial derivative $\partial_\rho \varphi$ is continuous and uniformly bounded in $(t,y, z, u, \mu)$ for $\varphi=f,g,h$. In particular, we require $\|\partial_z g(t, y,z,u,\mu)\|<\alpha_1\in(0,1)$. 
		
	\medskip
	
		\item[(H2)]The functions $f$, $g$ and  $h$ are L-differentiable with respect to $\mu$.  Moreover, for ${\nu}=\mu_y,\mu_z,\mu_u$, the L-derivative $\partial_\nu \varphi$  is continuous with $L^2$-norm being  uniformly bounded  in $(t, y, z, u, \mu)$ for $\varphi=f, g, h$.  In particular, we require 
		\[\int_{\R^n\times \R^{n\times l}\times \R^k}\big\|\partial_{\mu_z}g(t, y,z,u, \mu)(y',z',u')\big\|^2d\mu(y',z',u')<\alpha_2\in (0, 1-\alpha_1).\]
		
		\medskip
		
		\item[(H3)]The function $\Phi$ is differentiable with respect to $y$ and L-differentiable with respect to $\mu_y$, and moreover $\partial_y \Phi(y,\mu)$ and $\partial_{\mu_y} \Phi(y,\mathcal L(Y))(Y)$ are jointly continuous.  
	\end{enumerate}
\begin{remark}\label{lip}	
	Note that, if $f$, $g$  are continuously differentiable with uniformly bounded partial derivatives as assumed in (H1) and (H2), we can deduce that $f$, $g$ are Lipschitz in $(y,z,u)$ and $\mu$. Precisely, there exists a constant $C$ and $0<\alpha_1,\alpha_2<1$ with $\alpha_1+\alpha_2<1$ such that for all $ y,y'\in\R^n$, $z,z'\in\R^{n\times l}$, $u,u'\in\R^k$, $\mu=\mathcal L(y,z,u),\mu'=\mathcal L(y',z',u')$,
	\begin{align*}
		&\big|f(t,y,z,u,\mu)-f(t,y',z',u',\mu')\big|^2\\&\leq C\Big(|y-y'|^2+\|z-z'\|^2+|u-u'|^2+\E\big[|y-y'|^2+\|z-z'\|^2+|u-u'|^2\big]\Big),
		\end{align*}
	and
	\begin{align*}
			&\big\|g(t,y,z,u,\mu)-g(t,y',z',u',\mu')\big\|^2\\&\leq C\Big(|y-y'|^2+|u-u'|^2+\E\big[|y-y'|^2+|u-u'|^2\big]\Big)+\Big(\alpha_1\|z-z'\|^2+\alpha_2\E\big[\|z-z'\|^2\big]\Big). 
		\end{align*}
\end{remark}

The following result borrowed from \cite{lx22} provides the existence and uniqueness for the solution of  \eqref{state-e}. 
\begin{theorem}\label{exist-unique-state}
	Under the Assumptions (H1) and (H2), for any fixed $u=(u_t)_{0\leq t\leq T}\in \mathcal U$, there exists a unique solution $(y^u,z^u)\in  S^2_\mathcal F([0,T];\R^n)\times L^2_\mathcal F([0,T];\R^{n\times l})$ to~ \eqref{state-e}.
	\end{theorem}

 Let $u\in \mathcal U$ be an optimal control, i.e., $J(u)=\inf\limits_{v\in\mathcal U}J(v)$, and $(y,z)$ be the corresponding state process.  We shall introduce some notations that will be used in the sequel.

  Recalling that $\mathcal U=L^2_{\mathcal F}([0,T]; U)$ with $U$ being a convex set of $\R^k$, we have  $u^\e:=u+\varepsilon v\in \mathcal U$ for  $0\le\e\le 1$ and all $v=\bar u-u$ with $\bar u\in \mathcal U$. Let $(y^\varepsilon, z^\varepsilon)$ denote the solution of \eqref{state-e} with $u=u^\varepsilon$. We shall take the following abbreviated notations
\begin{equation}\label{e:notations}
\begin{aligned}
&\theta_t=(y_t,z_t,u_t,\mathcal L(y_t,z_t,u_t)),\,\,\theta_t^\varepsilon=(y_t^\varepsilon, z_t^\varepsilon,u_t^\varepsilon,\mathcal L(y^\varepsilon_t,z^\varepsilon_t,u^\varepsilon_t)),\\
&y^\lambda_t=y_t+\lambda(y_t^\varepsilon-y_t),\,\,z^\lambda_t=z_t+\lambda(z_t^\varepsilon-z_t),\,\,u^\lambda_t=u_t+\lambda\varepsilon v_t,\\
&\theta^\lambda_t=(y^\lambda_t,z^\lambda_t,u^\lambda_t,\mathcal L(y^\lambda_t,z^\lambda_t,u^\lambda_t)).
\end{aligned}
\end{equation}

 Let $(\widetilde{\Omega},\widetilde{\mathcal F},\widetilde P)$ be a copy of $(\Omega,\mathcal F,P)$. For a random variable $X$ defined on $(\Omega,\mathcal F,P)$, we denote by $\widetilde{X}$ its copy on $\widetilde{\Omega}$.   For any integrable random variable  $\xi$  on  the probability space $(\Omega\times \widetilde{\Omega},\mathcal F\times \widetilde{\mathcal F}, P\otimes \widetilde P)$,  we denote
\begin{equation}\label{e:notations1}
\E\big[\xi(\omega,\widetilde{\omega})\big]=\int_{\Omega}\xi(\omega,\widetilde{\omega})P(d\omega) \text{ and } \widetilde{\E}\big[\xi(\omega,\widetilde{\omega})\big]=\int_{\widetilde{\Omega}}\xi(\omega,\widetilde{\omega})P(d\widetilde{\omega}).
\end{equation}

With the above notations in mind, we introduce the following linear backward doubly stochastic differential equation
\begin{equation}\label{variational-e}
	\left\{
	\begin{aligned}
		-dK_t=&\Big\{\partial_yf(t,\theta_t)K_t+\partial_zf(t,\theta_t)L_t+\partial_uf(t,\theta_t)v_t+\widetilde \E\big[\partial_{\mu_y}f(t,\theta_t)(\widetilde y_t,\widetilde z_t,\widetilde u_t)\widetilde K_t\big]\\&\hspace{1em}+\widetilde \E\big[\partial_{\mu_z}f(t,\theta_t)(\widetilde y_t,\widetilde z_t,\widetilde u_t)\widetilde L_t\big]+\widetilde \E\big[\partial_{\mu_u}f(t,\theta_t)(\widetilde y_t,\widetilde z_t,\widetilde u_t)\widetilde v_t\big]\Big\}dt\\
		&+\Big\{\partial_yg(t,\theta_t)K_t+\partial_zg(t,\theta_t)L_t+\partial_ug(t,\theta_t)v_t+\widetilde \E\big[\partial_{\mu_y}g(t,\theta_t)(\widetilde y_t,\widetilde z_t,\widetilde u_t)\widetilde K_t\big]\\&\quad +\widetilde \E\big[\partial_{\mu_z}g(t,\theta_t)(\widetilde y_t,\widetilde z_t,\widetilde u_t)\widetilde L_t\big]+\widetilde \E\big[\partial_{\mu_u}g(t,\theta_t)(\widetilde y_t,\widetilde z_t,\widetilde u_t)\widetilde v_t\big]\Big\}d\overleftarrow B_t\\&-L_tdW_t,\,\,\,t\in[0,T],\\
		K_T=&0,
		\end{aligned}
	\right.
\end{equation}
to which there exists a unique solution $(K,L)\in  S^2_\mathcal F([0,T];\R^n)\times L^2_\mathcal F([0,T];\R^{n\times l})$ by Theorem~\ref{exist-unique-state}. 

\begin{proposition}\label{estimate}
	Let assumptions (H1)-(H3) hold. Then, we have
	\begin{align*}
		\lim\limits_{\varepsilon\to 0}\E\left[\sup\limits_{0\leq t\leq T}\Big|\frac{y^\varepsilon_t-y_t}{\varepsilon}-K_t\Big|^2\right]=0 ~\text{ and }~ \lim\limits_{\varepsilon\to 0}\E\left[\int_0^T\Big\|\frac{z^\varepsilon_t-z_t}{\varepsilon}-L_t\Big\|^2dt\right]=0.
	\end{align*}
\end{proposition}

\begin{proof}
Denote
\begin{align}\label{e:hat-y-e}
	\hat y_t^\varepsilon=\frac{y^\varepsilon_t-y_t}{\varepsilon}-K_t ~ \text{ and } ~\hat z_t^\varepsilon=\frac{z^\varepsilon_t-z_t}{\varepsilon}-L_t.
\end{align}
Then, by \eqref{state-e} and \eqref{variational-e},  $(\hat y^\varepsilon_t,\hat z^\varepsilon_t)_{0\leq t\leq T}$ solves the following equation,
\begin{equation}\label{e:y-hat}
	\left\{
	\begin{aligned}
		-d\widehat y^\varepsilon_t=&\Big\{\tfrac{1}{\varepsilon}\big[f(t,\theta^\varepsilon_t)-f(t,\theta_t)\big]-\partial_yf(t,\theta_t)K_t-\partial_zf(t,\theta_t)L_t-\partial_uf(t,\theta_t)v_t\\& \quad -\widetilde \E\big[\partial_{\mu_y}f(t,\theta_t)(\widetilde y_t,\widetilde z_t,\widetilde u_t)\widetilde K_t\big]-\widetilde \E\big[\partial_{\mu_z}f(t,\theta_t)(\widetilde y_t,\widetilde z_t,\widetilde u_t)\widetilde L_t\big]\\
		&\quad -\widetilde \E\big[\partial_{\mu_u}f(t,\theta_t)(\widetilde y_t,\widetilde z_t,\widetilde u_t)\widetilde v_t\big]\Big\}dt\\
		&+\Big\{\tfrac{1}{\varepsilon}\big[g(t,\theta^\varepsilon_t)-g(t,\theta_t)\big]-\partial_yg(t,\theta_t)K_t-\partial_zg(t,\theta_t)L_t-\partial_ug(t,\theta_t)v_t\\
		&\qquad -\widetilde \E\big[\partial_{\mu_y}g(t,\theta_t)(\widetilde y_t,\widetilde z_t,\widetilde u_t){\widetilde K}_t\big]-\widetilde \E\big[\partial_{\mu_z}g(t,\theta_t)(\widetilde y_t,\widetilde z_t,\widetilde u_t)\widetilde L_t\big]\\
		& \qquad  -\widetilde \E\big[\partial_{\mu_u}g(t,\theta_t)(\widetilde y_t,\widetilde z_t,\widetilde u_t)\widetilde v_t\big]\Big\}d\overleftarrow B_t-\widehat z^\varepsilon_tdW_t,\,\,\,t\in[0,T],\\
		\hat y^\varepsilon_T=&0.
	\end{aligned}
	\right.
\end{equation}

 Using notations \eqref{e:notations} and \eqref{e:notations1}, some algebraic work shows that \eqref{e:y-hat} can be written as
 \begin{equation}
	\left\{
	\begin{aligned}
		-d\hat y^\varepsilon_t=&\bigg[\int_0^1\Big\{\partial_yf(t,\theta_t^\lambda) {\widehat y_t^\varepsilon}+\widetilde \E\big[\partial_{\mu_y}f(t,\theta_t^\lambda)(\widetilde{ y^\lambda_t},\widetilde{ z^\lambda_t},\widetilde {u^\lambda_t})\widetilde {\widehat y^\varepsilon_t}\big]\\&\hspace{2em}+\partial_zf(t,\theta_t^\lambda) {\widehat z^\varepsilon_t}+\widetilde \E\big[\partial_{\mu_z}f(t,\theta^\lambda_t)(\widetilde{ y^\lambda_t},\widetilde {z^\lambda_t},\widetilde {u^\lambda_t})\widetilde {\widehat z^\varepsilon_t}\big]+F^{\lambda,\varepsilon}_t\Big\}d\lambda \bigg]dt\\
		&+\bigg[\int_0^1\Big\{\partial_yg(t,\theta_t^\lambda){\widehat y_t^\varepsilon}+\widetilde \E\big[\partial_{\mu_y}g(t,\theta_t^\lambda)(\widetilde{ y^\lambda_t},\widetilde {z^\lambda_t},\widetilde {u^\lambda_t})\widetilde {\widehat y^\varepsilon_t}\big]\\&\hspace{3em}+\partial_zg(t,\theta_t^\lambda) {\widehat z^\varepsilon_t}+\widetilde \E\big[\partial_{\mu_z}g(t,\theta^\lambda_t)(\widetilde{ y^\lambda_t},\widetilde {z^\lambda_t},\widetilde {u^\lambda_t})\widetilde {\widehat z^\varepsilon_t}\big]+G^{\lambda,\varepsilon}_t\Big\}d\lambda \bigg]d\overleftarrow B_t\\
		&-\widehat z^\varepsilon_tdW_t,\,\,\,t\in[0,T],\\
		\hat y^\varepsilon_T=&0.
	\end{aligned}
	\right.
\end{equation}
where 
\begin{align*}
	F^{\lambda,\varepsilon}_t:=&\big(\partial_yf(t,\theta_t^\lambda)-\partial_yf(t,\theta_t)\big)K_t+\widetilde\E\Big[\big(\partial_{\mu_y}f(t,\theta_t^\lambda)(\widetilde{ y^\lambda_t},\widetilde {z^\lambda_t},\widetilde {u^\lambda_t})-\partial_{\mu_y}f(t,\theta_t)(\widetilde y_t,\widetilde z_t,\widetilde u_t)\big)\widetilde K_t\Big]\\
	+&\big(\partial_zf(t,\theta_t^\lambda)-\partial_zf(t,\theta_t)\big)L_t+\widetilde\E\Big[\big(\partial_{\mu_z}f(t,\theta_t^\lambda)(\widetilde{ y^\lambda_t},\widetilde {z^\lambda_t},\widetilde {u^\lambda_t})-\partial_{\mu_z}f(t,\theta_t)(\widetilde y_t,\widetilde z_t,\widetilde u_t)\big)\widetilde L_t\Big]\\
	+&\big(\partial_uf(t,\theta_t^\lambda)-\partial_uf(t,\theta_t)\big)v_t+\widetilde \E\Big[\big(\partial_{\mu_u}f(t,\theta^\lambda_t)(\widetilde{ y^\lambda_t},\widetilde {z^\lambda_t},\widetilde {u^\lambda_t})-\partial_{\mu_u}f(t,\theta_t)(\widetilde y_t,\widetilde z_t,\widetilde u_t)\big) \widetilde v_t \Big],
\end{align*}
and $G^{\lambda, \varepsilon}_t$ is of the same form as  $F^{\lambda,\e}_t$ with $f$ replaced by $g$. 

Applying It\^o's formula \eqref{e:ito} to $\left|\widehat y_t^\varepsilon\right|^2$, we have
\begin{align*}
	&\E\big[|\widehat y_t^\varepsilon|^2\big]+\E\Big[\int_t^T\|\widehat z_s^\varepsilon\|^2ds \Big]\\
	=&2\E\int_t^T\Big<\int_0^1\big\{\partial_yf(s,\theta_s^\lambda) {\widehat y_s^\varepsilon}+\widetilde \E\big[\partial_{\mu_y}f(s,\theta_s^\lambda)(\widetilde{ y^\lambda_s},\widetilde {z^\lambda_s},\widetilde {u^\lambda_s})\widetilde {\widehat y^\varepsilon_s}\big]\\&\hspace{5em}+\partial_zf(s,\theta_s^\lambda) {\widehat z^\varepsilon_s}+\widetilde \E\big[\partial_{\mu_z}f(s,\theta^\lambda_s)(\widetilde{ y^\lambda_s},\widetilde {z^\lambda_s},\widetilde {u^\lambda_s})\widetilde {\widehat z^\varepsilon_s}\big]+F^{\lambda, \varepsilon}_s\big\}d\lambda,\widehat y^\varepsilon_s\Big>ds\\&\hspace{0.5em}+\E\int_t^T\Big\|\int_0^1\big\{\partial_yg(s,\theta_s^\lambda){\widehat y_s^\varepsilon}+\widetilde \E\big[\partial_{\mu_y}g(s,\theta_s^\lambda)(\widetilde{ y^\lambda_s},\widetilde {z^\lambda_s},\widetilde {u^\lambda_s})\widetilde {\widehat y^\varepsilon_s}\big]\\&\hspace{6.2em}+\partial_zg(s,\theta_s^\lambda) {\widehat z^\varepsilon_s}+\widetilde \E\big[\partial_{\mu_z}g(s,\theta^\lambda_s)(\widetilde{ y^\lambda_s},\widetilde {z^\lambda_s},\widetilde {u^\lambda_s})\widetilde {\widehat z^\varepsilon_s}\big]+G^{\lambda, \varepsilon}_s\big\}d\lambda\Big\|^2ds.
\end{align*}
The uniform boundedness of the partial derivatives of $f$ and $g$ as assumed in (H1) and (H2) yields   
\begin{align*}
	\E\big[|\widehat y_t^\varepsilon|^2\big]+\E\Big[\int_t^T\|\widehat z_s^\varepsilon\|^2ds \Big]\leq C\E\Big[\int_t^T|\widehat y^\varepsilon_s|^2ds\Big]+\E\int_t^T\int_0^1\big\{\big| F^{\lambda,\varepsilon}_s\big|^2+\big\| G^{\lambda, \varepsilon}_s\big\|^2\big\}d\lambda ds.
\end{align*}
To get the desired result, in view of Gronwall's lemma, it suffices to show 
\begin{equation}\label{estimate-F}
	\lim\limits_{\varepsilon\to0}\E\Big[\int_0^T\int_0^1\big|F_t^{\lambda,\varepsilon}\big|^2d\lambda dt\Big]=0,
\end{equation}
and
\begin{equation}\label{estimate-G}
	\lim\limits_{\varepsilon\to 0}\E\Big[\int_0^T\int_0^1\big\|G_t^{\lambda, \varepsilon}\big\|^2d\lambda dt\Big]=0.
\end{equation}
We shall prove \eqref{estimate-F} below, and \eqref{estimate-G} can be proved in the same way and thus omitted. 
By H\"older's inequality, we have
\begin{align*}
	&\E\Big[\int_0^T\int_0^1\big|F_t^{\lambda,\varepsilon}\big|^2d\lambda dt\Big]\\&\leq C\E\int_0^T\Big\{\int_0^1\big\{|\partial_yf(t,\theta_t^\lambda)-\partial_yf(t,\theta_t)|^2|K_t|^2+\|\partial_zf(t,\theta_t^\lambda)-\partial_zf(t,\theta_t)\|^2\|L_t\|^2\big\}d\lambda\\&\hspace{2.5em}\qquad+\int_0^1\widetilde\E\big[\big|\partial_{\mu_y}f(t,\theta_t^\lambda)(\widetilde{ y^\lambda_t},\widetilde {z^\lambda_t},\widetilde {u^\lambda_t})-\partial_{\mu_y}f(t,\theta_t)(\widetilde y_t,\widetilde z_t,\widetilde u_t)\big|^2\big]\widetilde\E\big[|\widetilde K_t|^2\big]d\lambda\\&\hspace{2.5em}\qquad+\int_0^1\widetilde\E\big[\big\|\partial_{\mu_z}f(t,\theta_t^\lambda)(\widetilde{ y^\lambda_t},\widetilde {z^\lambda_t},\widetilde {u^\lambda_t})-\partial_{\mu_z}f(t,\theta_t)(\widetilde y_t,\widetilde z_t,\widetilde u_t)\big\|^2\big]\widetilde\E\big[\|\widetilde L_t\|^2\big]d\lambda\\&\hspace{2.5em}\qquad +\int_0^1\widetilde\E\big[\big|\partial_{{\mu_u}}f(t,\theta_t^\lambda)(\widetilde{ y^\lambda_t},\widetilde {Z^\lambda_t},\widetilde {u^\lambda_t})-\partial_{\mu_y}f(t,\theta_t)(\widetilde y_t,\widetilde z_t,\widetilde u_t)\big|^2\big]\widetilde\E\big[|\widetilde v_t|^2\big]d\lambda\\&\hspace{2.5em}\qquad+\int_0^1\big|\partial_uf(t,\theta_t^\lambda)-\partial_uf(t,\theta_t)\big|^2|v_t|^2d\lambda	\Big\}dt.
 \end{align*} 
Due to the continuity and uniform boundedness assumed in (H1) and (H2) for the partial derivatives, we can apply the dominated convergence theorem to prove \eqref{estimate-F}. The proof is concluded. 
\end{proof}

The differentiability of  the cost functional $J(\cdot)$ proved in the following result will be used in the derivation of the variational inequality in Section \ref{sec:adjoint-eq}.

\begin{proposition}\label{expan-J}
	Under conditions (H1)-(H3), the cost functional $J(\cdot)$ defined by \eqref{cost-e} is Gateaux differentiable, and the derivative  at $u$ in the direction $v$ is  given by 
		\begin{align}\label{e:J-derivative}
		\frac{d}{d\varepsilon}J(u+\varepsilon v)\Big|_{\varepsilon=0}=&\E\int_0^T\Big\{\big<\partial_yh(t,\theta_t),K_t\big>+\big<\partial_zh(t,\theta_t),L_t\big>+\big<\partial_uh(t,\theta_t),v_t\big>\notag\\&\hspace{2.5em}+\widetilde \E\big[\big<\partial_{\mu_y}h(t,\theta_t)(\widetilde y_t,\widetilde z_t,\widetilde u_t),\widetilde K_t\big>\big]+\widetilde \E\big[\big<\partial_{\mu_z}h(t,\theta_t)(\widetilde y_t,\widetilde z_t,\widetilde u_t),\widetilde L_t\big>\big]\notag\\&\hspace{2.5em}+\widetilde \E\big[\big<\partial_{{\mu_u}}h(t,\theta_t)(\widetilde y_t,\widetilde z_t,\widetilde u_t),\widetilde v_t\big>\big]\Big\}dt\notag\\& +\E\Big[\big<\partial_y\Phi(y_0,\mathcal L(y_0)),K_0\big>+\widetilde \E\big[\big<\partial_{\mu_y}\Phi(y_0,\mathcal L(y_0))(\widetilde y_0),\widetilde K_0\big>\big]\Big].
	\end{align}
\end{proposition}

\begin{proof}
	By the definition \eqref{cost-e} of $J$ and the notations \eqref{e:notations}, we have
	\begin{align}
		\frac{J(u^\varepsilon)-J(u)}{\varepsilon}&=\frac{1}{\varepsilon} \Big[\E\int_0^T\big\{h(t,\theta^\varepsilon_t)-h(t,\theta_t)\big\}dt\Big]+\frac{1}{\varepsilon}\E\Big[\Phi(y_0^\varepsilon,\mathcal L(y_0^\varepsilon))-\Phi(y_0,\mathcal L(y_0))\Big]\notag\\&=:I_1+I_2.\label{e:J-J}
	\end{align}
	For the  term $I_1$,  Taylor's first-order expansion yields
	\begin{align*}
		I_1=&\E\int_0^T\Big\{\big<\partial_yh(t,\theta_t),K_t\big>+\big<\partial_zh(t,\theta_t),L_t\big>+\big<\partial_uh(t,\theta_t),v_t\big>\\&\hspace{3em}+\widetilde \E\big[\big<\partial_{\mu_y}h(t,\theta_t)(\widetilde{y_t},\widetilde{z_t},\widetilde{u_t}),\widetilde K_t\big>\big]+\widetilde \E\big[\big<\partial_{\mu_z}h(t,\theta_t)(\widetilde{y_t},\widetilde{z_t},\widetilde{u_t}),\widetilde L_t\big>\big]\\&\hspace{3em}+\widetilde \E\big[\big<\partial_{{\mu_u}}h(t,\theta_t)(\widetilde{y_t},\widetilde{z_t},\widetilde{u_t}),\widetilde v_t\big>\big]\Big\}dt+\E\int_0^T\rho_t^\varepsilon dt,
	\end{align*}
with 
\begin{align*}
	\rho^\varepsilon_t:=&\int_0^1\big\{\big<\partial_yh(t,\theta^\lambda_t),\widehat y_t^\varepsilon\big>+\big<\partial_yh(t,\theta^\lambda_t)-\partial_y h(t,\theta_t),K_t\big>\big\}d\lambda\\&+\int_0^1\big\{\big<\partial_zh(t,\theta^\lambda_t),\widehat z_t^\varepsilon\big>+\big<\partial_zh(t,\theta^\lambda_t)-\partial_z h(t,\theta_t),L_t\big>\big\}d\lambda\\&+\int_0^1\big\{\widetilde \E\big[\big<\partial_{\mu_y}h(t,\theta^\lambda_t)(\widetilde{y^\lambda_t},\widetilde{z^\lambda_t},\widetilde{u^\lambda_t}),\widetilde{\widehat y_t^\varepsilon}\big>\big]+\widetilde \E\big[\big<\partial_{\mu_z}h(t,\theta^\lambda_t)(\widetilde{y^\lambda_t},\widetilde{z^\lambda_t},\widetilde{u^\lambda_t}),\widetilde{\widehat z_t^\varepsilon}\big>\big]\big\}d\lambda
\\	&+\int_0^1\widetilde\E\big[\big<\partial_{\mu_y}h(t,\theta^\lambda_t)(\widetilde{y^\lambda_t},\widetilde{z^\lambda_t},\widetilde{u^\lambda_t})-\partial_{\mu_y}h(t,\theta_t)(\widetilde y_t,\widetilde z_t,\widetilde u_t),\widetilde K_t\big>\big]d\lambda\\
&+\int_0^1\widetilde\E\big[\big<\partial_{\mu_z}h(t,\theta^\lambda_t)(\widetilde{y^\lambda_t},\widetilde{z^\lambda_t},\widetilde{u^\lambda_t})-\partial_{\mu_z}h(t,\theta_t)(\widetilde y_t,\widetilde z_t,\widetilde u_t),\widetilde L_t\big>\big]d\lambda\\
&+\int_0^1\widetilde\E\big[\big<\partial_{{\mu_u}}h(t,\theta^\lambda_t)(\widetilde{y^\lambda_t},\widetilde{z^\lambda_t},\widetilde{u^\lambda_t})-\partial_{{\mu_u}}h(t,\theta_t)(\widetilde y_t,\widetilde z_t,\widetilde u_t),\widetilde v_t\big>\big]d\lambda\\&+\int_0^1\big<\partial_uh(t,\theta_t^\lambda)-\partial_uh(t,\theta_t),v_t\big>d\lambda,
\end{align*}
where we recall that $\hat y_t^\e, \hat z_t^\e$ are given in \eqref{e:hat-y-e}. 

Note that by (H1) and (H2), the partial derivatives of $h$ are jointly continuous and uniformly bounded. Combining this fact with  Proposition \ref{estimate},  we can apply dominated convergence theorem to get  \[\lim\limits_{\varepsilon\to 0}\E\int_0^T\rho_t^\varepsilon dt=0.\]
 The term $I_2$ in \eqref{e:J-J} can be analyzed in a similar way. The proof is concluded. 
\end{proof}

\subsection{On necessity of the condition} \label{sec:adjoint-eq}

In this subsection, we present our main result of stochastic maximum principle, which is a necessary condition for an optimal control. 

Let $H: [0,T]\times \R^n\times \R^{n\times l} \times\R^k\times \mathcal P_2(\R^n\times \R^{n\times l}\times \R^k)\times\R^n\times\R^{n\times d} \to \R$  denote the Hamiltonian given by
\begin{align}\label{Hamiltonian}
	H(t,y,z,u,\mu,p,q):=\big<f(t,y,z,u,\mu),p\big>+\big<g(t,y,z,u,\mu),q\big>+h(t,y,z,u,\mu).
\end{align}
Consider the following adjoint equation
\begin{equation}\label{e:pq}
	\left\{
\begin{aligned}
 dp_t=&\Big[\partial_yH(t,\theta_t,p_t,q_t)+\widetilde \E\big[\partial_{\mu_y}H(t,\widetilde\theta_t,\widetilde p_t,\widetilde q_t)(y_t,z_t,u_t)\big]\Big]dt\\&+\Big[\partial_zH(t,\theta_t,p_t,q_t)+\widetilde \E\big[\partial_{\mu_z}H(t,\widetilde\theta_t,\widetilde p_t,\widetilde q_t)(y_t,z_t,u_t)\big]\Big]dW_t\\&-q_td\overleftarrow B_t, \,\,\,t\in[0,T],\\
 p_0=&\partial_y\Phi(y_0,\mathcal L(y_0))+\widetilde\E\big[\partial_{\mu_y}\Phi(\widetilde y_0,\mathcal L(y_0))(y_0)\big],
	\end{aligned}
	\right.
	\end{equation}
where we have used notations in \eqref{e:notations}.  Recalling the equation \eqref{variational-e} of $(K,L)$,  applying It\^o's formula to $\big<p_t,K_t\big>$ from $0$ to $T$ and taking expectation, we can get 
\begin{align*}
	\E\big[\big<p_0,K_0\big>\big]=&\E\int_0^T\Big\{\big<\partial_u^\text{T}f(t,\theta_t)p_t+\widetilde\E\big[\partial^\text{T}_{\mu_u} f(t,\widetilde \theta_t)(y_t,z_t,u_t)\widetilde p_t\big]\\&\hspace{3.5em}+\partial_u^\text{T}g(t,\theta_t)q_t+\widetilde\E\big[\partial^\text{T}_{\mu_u} g(t,\widetilde \theta_t)(y_t,z_t,u_t)\widetilde q_t\big],v_t\big>\\&\hspace{3em}-\big<\partial_yh(t,\theta_t)+\widetilde\E\big[\partial_{\mu_y}h(t,\widetilde \theta_t)(y_t,z_t,u_t)\big],K_t\big>\\&\hspace{3em}-\big<\partial_zh(t,\theta_t)+\widetilde\E\big[\partial_{\mu_z}h(t,\widetilde \theta_t)(y_t,z_t,u_t)\big],L_t\big>\Big\}dt.
\end{align*}
Note that $\E\big[\big<p_0,K_0\big>\big]$ is the sum of the last two terms on the right-hand side of \eqref{e:J-derivative}.  Plugging this expression into \eqref{e:J-derivative}, we get
\begin{align*}
	&\frac{d}{d\varepsilon}J(u+\varepsilon v)\Big|_{\varepsilon=0}\\=&\E\int_0^T\Big<\partial_u^\text{T}f(t,\theta_t)p_t+\widetilde\E\big[\partial^\text{T}_{\mu_u} f(t,\widetilde \theta_t)(y_t,z_t,u_t)\widetilde p_t\big]+\partial_u^\text{T}g(t,\theta_t)q_t\\
	&\qquad\qquad+\widetilde\E\big[\partial^\text{T}_{\mu_u} g(t,\widetilde \theta_t)(y_t,z_t,u_t)\widetilde q_t\big]+\partial_uh(t,\theta_t)+\widetilde \E\big[\partial_{\mu_u}h(t,\widetilde \theta_t)(y_t,z_t,u_t)\big],v_t\Big>dt.
\end{align*}
Using the Hamiltonian $H$ given by \eqref{Hamiltonian}, we can write
\begin{align}\label{e:J-d}
	\frac{d}{d\varepsilon}J(u+\varepsilon v)\Big|_{\varepsilon=0}=\E\int_0^T\Big<\partial_uH(t,\theta_t,p_t,q_t)+\widetilde\E\big[\partial_{\mu_u} H(t,\widetilde \theta_t,\widetilde p_t,\widetilde q_t)(y_t,z_t,u_t)\big],v_t\Big>dt.
	\end{align}

Now we are ready to derive our main result, the stochastic maximum principle.
\begin{theorem}\label{nmp}
We assume conditions (H1)-(H3)  for the control problem \eqref{state-e}-\eqref{cost-e}. Suppose that $u=(u_t)_{0\leq t\leq T}\in\mathcal U$ is an  optimal control, $(y_t,z_t)_{0\leq t\leq T}$ is  the associated state process, and $(p_t,q_t)_{0\leq t\leq T}$ is the adjoint process satisfying \eqref{e:pq}. Then,  we have, for all $a\in U$,
	\begin{align}\label{variational inequality}
	\Big<\partial_uH(t,\theta_t,p_t,q_t)+\widetilde\E\big[\partial_{\mu_u} H(t,\widetilde \theta_t,\widetilde p_t,\widetilde q_t)(y_t,z_t,u_t)\big],a-u_t\Big>\geq 0, \,\,\,dt\otimes d\mathbb P \text{ a.s.},
	\end{align}
where $H$ is the Hamiltonian defined by \eqref{Hamiltonian}.
\end{theorem}
\begin{proof}
Given any admissible control $(\bar u_t)_{0\leq t\leq T}\in\mathcal U$, we denote $v_t=\bar u_t-u_t$. We  use the  perturbation $u^\varepsilon_t=u_t+\varepsilon v_t$. Since $u$ is optimal, i.e., $J(u)$ achieves the minimum, we have
\begin{align*}
	\frac{d}{d\varepsilon}J(u+\varepsilon v)\Big|_{\varepsilon=0}\geq 0.
\end{align*}
This together with \eqref{e:J-d} implies
\begin{align}\label{eq3.8}
	\E\int_0^T\Big<\partial_uH(t,\theta_t,p_t,q_t)+\widetilde\E\big[\partial_{\mu_u} H(t,\widetilde \theta_t,\widetilde p_t,\widetilde q_t)(y_t,z_t,u_t)\big],\bar u_t-u_t\Big>dt\geq 0.
\end{align}
We set an admissible control $(\bar u_t)_{0\leq t\leq T}$ as follows
\begin{align*}
	\bar u_s=\left\{
	\begin{array}{l}
		\alpha_s, \ \ \ s\in[t,t+\varepsilon],\\
		u_s, \ \ \ \text{otherwise},
	\end{array}
	\right.
\end{align*}
where $(\alpha_t)_{0\leq t\leq T}\in \mathcal{U}$. From \eqref{eq3.8}, we have
\begin{align}
	&\frac{1}{\varepsilon}\E\left[\int^{t+\varepsilon}_t\big<\partial_uH(s,\theta_s,p_s,q_s)+\widetilde\E\big[\partial_{\mu_u} H(s,\widetilde \theta_s,\widetilde p_s,\widetilde q_s)(y_s,z_s,u_s)\big],\alpha_s-u_s\big>ds\right]\geq0,
\end{align}
Letting $\varepsilon\rightarrow 0^+$, by Lebesgue differential theorem, we have for almost all $t$, 
\begin{align*}
	&\E\left[\big<\partial_uH(t,\theta_t,p_t,q_t)+\widetilde\E\big[\partial_{\mu_u} H(t,\widetilde \theta_t,\widetilde p_t,\widetilde q_t)(y_t,z_t,u_t)\big],\alpha_t-u_t\big>\right]\geq0.
\end{align*}
For  $A\in \mathcal F_t$,  we set $\alpha_t=a\textbf{1}_A+u_t\textbf{1}_{A^c}$ with $a\in U$. Thus, we have for almost all $t$, 
\begin{align*}
	\E\left[\big<\partial_uH(t,\theta_t,p_t,q_t)+\widetilde\E\big[\partial_{\mu_u} H(t,\widetilde \theta_t,\widetilde p_t,\widetilde q_t)(y_t,z_t,u_t)\big],a-u_t\big>\textbf{1}_A\right]\ge0.
\end{align*}
As   $A\in \mathcal F_t$ is chosen arbitrarily, the definition of conditional expectation leads to,  for almost all $t$, 
	\begin{align*}
	&\E\left[\big<\partial_uH(t,\theta_t,p_t,q_t)+\widetilde\E\big[\partial_{\mu_u} H(t,\widetilde \theta_t,\widetilde p_t,\widetilde q_t)(y_t,z_t,u_t)\big],a-u_t\big>\big|\mathcal F_t\right]\\&=\big<\partial_uH(t,\theta_t,p_t,q_t)+\widetilde\E\big[\partial_{\mu_u} H(t,\widetilde \theta_t,\widetilde p_t,\widetilde q_t)(y_t,z_t,u_t)\big],a-u_t\big>\geq 0, \text{a.s. }
\end{align*}
This proves the desired result.
\end{proof}

\subsection{On sufficiency of the condition}\label{sufficiency}

 In this subsection, we prove a verification theorem which states that under proper conditions,  the maximum principle \eqref{variational inequality}  obtained in Theorem \ref{nmp} does yield an optimal control. 
  
\begin{theorem}\label{smp}
	Assume (H1)-(H3).  We further assume that the Hamiltonian $H$ given in \eqref{Hamiltonian} and $\Phi$ are convex in the sense 
	\begin{align*}
		&H(t,y',z',u',\mu',p,q)-H(t,y,z,u,\mu,p,q)\\&\geq \big<\partial_yH(t,y,z,u,\mu,p,q),y'-y\big>+\widetilde\E\big[\big<\partial_{\mu_y}H(t,y,z,u,\mu,p,q)(\widetilde Y,\widetilde Z,\widetilde U),\widetilde Y'-\widetilde Y\big>\big]\\&+\big<\partial_zH(t,y,z,u,\mu,p,q),z'-z\big>+\widetilde\E\big[\big<\partial_{\mu_z}H(t,y,z,u,\mu,p,q)(\widetilde Y,\widetilde Z,\widetilde U),\widetilde Z'-\widetilde Z\big>\big]\\&+\big<\partial_u H(t,y,z,u,\mu,p,q),u'-u\big>+\widetilde\E\big[\big<\partial_{{\mu_u}}H(t,y,z,u,\mu,p,q)(\widetilde Y,\widetilde Z,\widetilde U),\widetilde U'-\widetilde U\big>\big],
	\end{align*}
	and
	\begin{align*}
		\Phi(y',\mu_y')-\Phi(y,\mu_y)\geq \big<\partial_y\Phi(y,\mu_y),y'-y\big>+\widetilde\E\big[\big<\partial_{\mu_y}\Phi(y,\mu_y)(\widetilde Y),\widetilde Y'-\widetilde Y\big>\big],
	\end{align*}
	for all  $y,y'\in\R^n$, $z,z'\in\R^{n\times l}$, $u,u'\in\R^k$, $\mu,\mu'\in \mathcal P_2(\R^n\times\R^{n\times l}\times \R^k)$ with $\mu=(\mu_y,\mu_z, \mu_u)=\mathcal L(\widetilde Y,\widetilde Z,\widetilde U)$, $\mu'=(\mu_x', \mu_y', \mu_u')=\mathcal L(\widetilde Y',\widetilde Z',\widetilde U')$,  $p\in\R^n$ and $q\in \R^{n\times d}$. 
	
	Let $u=(u_t)_{0\leq t\leq T}\in \mathcal U$ be an admissible control, $(y_t,z_t)_{0\leq t\leq T}$ the state process and $(p_t,q_t)_{0\leq t\leq T}$ the adjoint process. Then, if \eqref{variational inequality} holds, $u$ is an optimal control.
\end{theorem}

\begin{proof}
 Recalling the definition of \eqref{cost-e} of $J$ and the notations \eqref{e:notations}, we have
	\begin{align*}
		J(v)-J(u)=\E\int_0^T\big\{h(t,\theta^v_t)-h(t,\theta_t)\big\}dt+\E\big[\Phi(y_0^v,\mathcal L(y_0^v))-\Phi(y_0,\mathcal L(y_0))\big],
	\end{align*}
	where we use the superscript $v$ to denote the processes associated to the control process $(v_t)_{0\leq t\leq T}\in\mathcal U$.
	It follows directly from the convexity of $H$ and $\Phi$ that
	\begin{align}\label{eq3.11}
		&h(t,\theta^v_t)-h(t,\theta_t)\notag\\
		&=H(t,\theta^v_t,p_t,q_t)-H(t,\theta_t,p_t,q_t)-\big<f(t,\theta_t^v)-f(t,\theta_t),p_t\big>-\big<g(t,\theta_t^v)-g(t,\theta_t),q_t\big>\notag\\
		&\geq  \big<\partial_yH(t,\theta_t,p_t,q_t),y^v_t-y_t\big>+\widetilde\E\big[\big<\partial_{\mu_y}H(t,\theta_t,p_t,q_t)(\widetilde y_t,\widetilde z_t,\widetilde u_t),\widetilde y^v_t-\widetilde y_t\big>\big]\notag\\&\hspace{0.3em}+\big<\partial_zH(t,\theta_t,p_t,q_t),z^v_t-z_t\big>+\widetilde\E\big[\big<\partial_{\mu_z}H(t,\theta_t,p_t,q_t)(\widetilde y_t,\widetilde z_t,\widetilde u_t),\widetilde z^v_t-\widetilde z_t\big>\big]\\&\hspace{0.3em}+\big<\partial_uH(t,\theta_t,p_t,q_t),v_t-u_t\big>+\widetilde\E\big[\big<\partial_{{\mu_u}}H(t,\theta_t,p_t,q_t)(\widetilde y_t,\widetilde z_t,\widetilde u_t),\widetilde v_t-\widetilde u_t\big>\big]\notag\\&\hspace{0.3em}-\big<f(t,\theta_t^v)-f(t,\theta_t),p_t\big>-\big<g(t,\theta_t^v)-g(t,\theta_t),q_t\big>\notag,
	\end{align}
	and 
	\begin{align}\label{eq3.12}
		&\E\big[\Phi(y_0^v,\mathcal L(y_0^v))-\Phi(y_0,\mathcal L(y_0))\big]\notag\\
		&\geq  \E\Big[\big<\partial_y\Phi(y_0,\mathcal L(y_0)),y^v_0-y_0\big>+\widetilde \E\big[\big<\partial_{\mu_y}\Phi(y_0,\mathcal L(y_0))(\widetilde y_0),\widetilde y^v_0-\widetilde y_0\big>\big]\Big]\\
		&=\E\Big[\big<\partial_y\Phi(y_0,\mathcal L(y_0))+\widetilde\E\big[\partial_{\mu_y}\Phi(\widetilde y_0,\mathcal L(y_0))(y_0)\big],y_0^v-y_0\big>\Big]\notag.
	\end{align}
	Applying It\^o's formula to $\big<p_t,y^v_t-y_t\big>$ yields that
	\begin{align}\label{eq3.13}
		&\E\Big[\big<\partial_y\Phi(y_0,\mathcal L(y_0))+\widetilde\E\big[\partial_{\mu_y}\Phi(\widetilde y_0,\mathcal L(y_0))(y_0)\big],y_0^v-y_0\big>\Big]\notag\\&=\E\int_0^T\Big\{\big<f(t,\theta^v_t)-f(t,\theta_t),p_t\big>+\big<g(t,\theta^v_t)-g(t,\theta_t),q_t\big>\notag\\&\hspace{3em}-\big<\partial_yH(t,\theta_t,p_t,q_t)+\widetilde \E\big[\partial_{\mu_y}H(t,\widetilde\theta_t,\widetilde p_t,\widetilde q_t)(y_t,z_t,u_t)\big],y^v_t-y_t\big>\\&\hspace{3em}-\big<\partial_zH(t,\theta_t,p_t,q_t)+\widetilde \E\big[\partial_{\mu_z}H(t,\widetilde\theta_t,\widetilde p_t,\widetilde q_t)(y_t,z_t,u_t)\big],z^v_t-z_t\big>\Big\}dt.\notag
	\end{align}
	Combining \eqref{eq3.11}-\eqref{eq3.13}, and using Fubini's theorem, we have
	\begin{align*}
		J(v)-J(u)\ge \E\int_0^T\Big<\partial_uH(t,\theta_t,p_t,q_t)+\widetilde\E\big[\partial_{{\mu_u}}H(t,\widetilde \theta_t,\widetilde p_t,\widetilde q_t)(y_t,z_t,u_t)\big],v_t-u_t\Big>dt
	\end{align*}
	Thus, if we assume 	\eqref{variational inequality},  we get 
	\begin{align*}
		J(v)-J(u)\geq 0.
	\end{align*}
Noting that $v\in \mathcal U$ is chosen arbitrarily, this implies that $u$ is an optimal control. The proof is concluded. 
\end{proof}

 \setcounter{equation}{0}

\section{Well-posedness of mean-field forward-backward doubly stochastic differential equations}\label{4}

Using the Hamiltonian $H$ given in \eqref{Hamiltonian}, the state equation \eqref{state-e} and the adjoint equation \eqref{e:pq} can be combined in the following system
\begin{equation}\label{Hamiltonian-system}
		\left\{
		\begin{aligned}
			-dy_t=&	\partial_p	H(t, \theta_t, p_t, q_t)
			dt+\partial_q H(t, \theta_t, p_t, q_t)d\overleftarrow B_t-z_tdW_t,\\
			dp_t=&\Big[\partial_yH(t,\theta_t,p_t,q_t)+\widetilde \E\big[\partial_{\mu_y}H(t,\widetilde\theta_t,\widetilde p_t,\widetilde q_t)(y_t,z_t,u_t)\big]\Big]dt\\&+\Big[\partial_zH(t,\theta_t,p_t,q_t)+\widetilde \E\big[\partial_{\mu_z}H(t,\widetilde\theta_t,\widetilde p_t,\widetilde q_t)(y_t,z_t,u_t)\big]\Big]dW_t\\&-q_td\overleftarrow B_t, \,\,\,t\in[0,T],\\
			y_T=&\xi,~ p_0=\partial_y\Phi(y_0,\mathcal L(y_0))+\widetilde\E\big[\partial_{\mu_y}\Phi(\widetilde y_0,\mathcal L(y_0))(y_0)\big],
				\end{aligned}
		\right.
	\end{equation}
where $\theta_t$ is given in \eqref{e:notations}.

 If $u_t$ is  a function of $y_t,z_t,p_t,q_t$ and their joint distribution (see, e.g,  the LQ case in Section \ref{lq}), the above system \eqref{Hamiltonian-system} can be written as a time-symmetric  FBDSDE introduced in Peng and Shi \cite{03ps} of mean-field type,  
	\begin{equation}\label{mf-bdsde}
	\left\{
	\begin{aligned}
		-dy_t=&f(t,y_t,p_t,z_t,q_t,\mathcal L(y_t,p_t,z_t,q_t))dt+g(t,y_t,p_t,z_t,q_t,\mathcal L(y_t,p_t,z_t,q_t))d\overleftarrow B_t-z_tdW_t,\\
		dp_t=&F(t,y_t,p_t,z_t,q_t,\mathcal L(y_t,p_t,z_t,q_t))dt+G(t,y_t,p_t,z_t,q_t,\mathcal L(y_t,p_t,z_t,q_t))dW_t-q_td\overleftarrow B_t,\\
			y_T=&\xi,\,\,	p_0=\Psi(y_0,\mathcal L(y_0)),
	\end{aligned}
	\right.
\end{equation}
where $\xi$ is  an $\mathcal F_T$-measurable random variable, $\Psi:\Omega\times \R^n\times\mathcal P_2(\R^n)\to \R^n$, and
$f, g, F, G$ are functions  from $\Omega\times [0,T]\times \R^n\times \R^n\times \R^{n\times l}\times \R^{n\times d}\times \mathcal P_2( \R^n\times \R^n\times \R^{n\times l}\times \R^{n\times d})$ to $\R^n, \R^{n\times d}, \R^n , \R^{n\times l}$, respectively. 
\begin{definition}\label{def:fbdsde}
	A quadruple of processes $(y,p,z,q) $ is called a solution of \eqref{mf-bdsde} if $(y,p,z,q)\in L^2_\mathcal F([0,T];\R^n\times\R^n\times \R^{n\times l}\times \R^{n\times d})$ and satisfies  \eqref{mf-bdsde}.
\end{definition}

Let  $\mathcal A(t,\zeta,\mu)=(-F,f,-G,g)(t,\zeta, \mu)$ where $\zeta=(y,p,z,q)$ and  $\mu$ stands for a generic element in $\mathcal P_2( \R^n\times \R^n\times \R^{n\times l}\times \R^{n\times d})$. Assume that for each $(\zeta,\mu)\in \R^n\times \R^n\times \R^{n\times l}\times \R^{n\times d}\times \mathcal P_2( \R^n\times \R^n\times \R^{n\times l}\times \R^{n\times d})$,  $\mathcal A(\cdot,\zeta,\mu)\in L^2_\mathcal F([0,T];\R^n\times \R^n\times\R^{n\times l}\times \R^{n\times d}) $ and that for  each $(y,\mu_y)\in \R^n\times \mathcal P_2(\R^n)$, $\Psi(y,\mu_y)\in L^2_{\mathcal F_0}(\R^n)$. For
almost all $(t,\omega)\in [0,T]\times \Omega$,  $ \zeta,\zeta'\in \R^n\times \R^n\times \R^{n\times l}\times \R^{n\times d}$, $\mu,\mu'\in\mathcal P_2( \R^n\times \R^n\times \R^{n\times l}\times \R^{n\times d})$,  $ y,y'\in\R^n$,  and $\mu_y,\mu_y'\in\mathcal P_2(\R^n)$, we assume the following conditions.
\begin{enumerate}
	\item [(A1)] There exists  $k_1>0$ such that 
	\begin{align*}
			|\mathcal A(t,\zeta,\mu)-\mathcal A(t,\zeta',\mu')|&\leq k_1\big(|\zeta-\zeta'|+W_2(\mu,\mu')\big),\\
	|\Psi(y,\mu_y)-\Psi(y',\mu'_y)|&\leq k_1\big(|y-y'|+W_2(\mu_y,\mu'_y)\big).
	\end{align*}
\item [(A2)] There exist  constants $k_2,k_3, k_4\ge0$ with $k_2+k_3>0$  and $k_3+k_4>0$ such that \begin{align*}
	&\E\big[\big<\mathcal A(t,\zeta,\mu)-\mathcal A(t,\zeta',\mu'),\zeta-\zeta'\big>\big]\\&\leq-k_2
	\big(\E[|y-y'|^2+\|z-z'\|^2]\big) -k_3\big(\E[|p-p'|^2+\|q-q'\|^2]\big),
\end{align*}
and  $$\E\big[\big<\Psi(y,\mu_y)-\Psi(y',\mu'_y),y-y'\big>\big]\geq k_4\E[|y-y'|^2].$$
Moreover, we make some further assumptions  if $k_2$ or $k_3$ is zero:  we  assume  
\begin{align*}
		\|g(t,\zeta,\mu)-g(t,\zeta',\mu')\|&\leq k_1(|y-y'|+|p-p'|+\|q-q'\|)+\lambda_1 \|z-z'\|\\&+k_1(\E[|y-y'|+|p-p'|+\|q-q'\|])+\lambda _2 \E[\|z-z'\|],
\end{align*} 
if $k_2=0$, 
and 
\begin{align*}
\|G(t,\zeta,\mu)-G(t,\zeta',\mu')\|&\leq k_1(|y-y'|+\|z-z'\|+|p-p'|)+\lambda_1 \|q-q'\|\\&+k_1(\E[|y-y'|+\|z-z'\|+|p-p'|])+\lambda _2 \E[\|q-q'\|],
\end{align*}
if $k_3=0$, where $\lambda_1,\lambda_2$ are nonnegative constants satisfying $\lambda_1+\lambda_2<1$.
\end{enumerate}

We shall employ the method of continuation introduced in \cite{99pw} (see also \cite{15byz}) to establish the existence of solution of \eqref{mf-bdsde}. Consider a family of  mean-field FBDSDEs parameterized by $\alpha\in[0,1]$,
\begin{equation}\label{parameter-e}
	\left\{
	\begin{aligned}
		-dy_t=&\big[f^\alpha(t,\zeta_t,\mu_t)+f_0(t)\big]dt+\big[g^\alpha(t,\zeta_t,\mu_t)+g_0(t)\big]d\overleftarrow B_t-z_tdW_t,\\
		dp_t=&\big[F^\alpha(t,\zeta_t,\mu_t)+F_0(t)\big]dt+\big[G^\alpha(t,\zeta_t,\mu_t)+G_0(t)\big]dW_t-q_td\overleftarrow B_t,\\
		y_T=&\xi,\,\,p_0=\Psi^\alpha(y_0,\mathcal L(y_0))+\Psi_0,
	\end{aligned}
	\right.
\end{equation}
where $\zeta_t=(y_t,p_t,z_t,q_t)$, $\mu_t=\mathcal L(y_t,p_t,z_t,q_t)$, $(F_0,f_0,G_0,g_0)\in L^2_{\mathcal F}([0,T];\R^{n}\times \R^n\times \R^{n\times l}\times \R^{n\times d})$, $\Psi_0\in L^2_{\mathcal F_0}(\R^n)$ and for any given $\alpha\in [0,1]$,
\begin{align*}
	&f^\alpha(t,\zeta_t,\mu_t)=\alpha f(t,\zeta_t,\mu_t)-(1-\alpha)k_3p_t,\,\,\, g^\alpha(t,\zeta_t,\mu_t)=\alpha g(t,\zeta_t,\mu_t)-(1-\alpha)k_3q_t,\\
	&F^\alpha(t,\zeta_t,\mu_t)=\alpha F(t,\zeta_t,\mu_t)+(1-\alpha)k_2y_t,\,\,\, G^\alpha(t,\zeta_t,\mu_t)=\alpha G(t,\zeta_t,\mu_t)+(1-\alpha)k_2z_t,\\
	&\Psi^\alpha(y_0,\mathcal L(y_0))=\alpha\Psi(y_0,\mathcal L(y_0))+(1-\alpha)y_0.
\end{align*}

When $\alpha=0$,  equation \eqref{parameter-e} is reduced to
\begin{equation}\label{alpha=0}
	\left\{
	\begin{aligned}
		-dy_t=&[-k_3p_t+f_0(t)]dt+[-k_3q_t+g_0(t)]d\overleftarrow B_t-z_tdW_t,\\
		dp_t=&[k_2y_t+F_0(t)]dt+[k_2z_t+G_0(t)]dW_t-q_td\overleftarrow B_t,\\
		y_T=&\xi,\,\,p_0=y_0+\Psi_0.
	\end{aligned}
	\right.
\end{equation}
The  existence and uniqueness of the solution of  equation \eqref{alpha=0} have been obtained in \cite[Proposition 3.6]{10hpw}.   The following lemma is the key ingredient of the continuation method, which says that, if \eqref{parameter-e} has a solution for some $\alpha_0\in [0,1)$,  it also has a solution for $\alpha\in [\alpha_0,\alpha_0+\delta_0]$, where $\delta_0$ is a constant independent of  $\alpha_0$.
\begin{lemma}\label{pertubation}
	Under (A1)-(A2), we assume that there exists a constant $ \alpha_0\in[0,1)$ such that given any $(F_0,f_0,G_0,g_0)\in L^2_{\mathcal F}([0,T];\R^{n}\times \R^n\times \R^{n\times l}\times \R^{n\times d})$ and $\Psi_0\in L^2_{\mathcal F_0}(\R^n)$, $\xi\in L^2_{\mathcal F_T}(\R^n)$, 
	 equation \eqref{parameter-e} with $\alpha=\alpha_0$ has a unique solution. Then,  there exists a constant $\delta_0\in (0,1)$ which only depends on
	$k_1,k_2,k_3,k_4,\lambda_1,\lambda_2$ and $T$, such that for any $\alpha\in [\alpha_0,\alpha_0+\delta_0]$, equation \eqref{parameter-e}  has a unique solution.
\end{lemma}
\begin{proof} 
	By the assumption,  for each
	$\overline \zeta=(\overline y,\overline p,\overline z,\overline q)\in L^2_\mathcal F([0,T];\R^n\times \R^n\times\R^{n\times l}\times \R^{n\times d})$ with law $\overline\mu\in \mathcal P_2(\R^n\times \R^n\times\R^{n\times l}\times \R^{n\times d})$, there exists a unique quadruple $\zeta=(y,p,z,q)\in L^2_\mathcal F([0,T];\R^n\times \R^n\times\R^{n\times l}\times \R^{n\times d})$ satisfying, for  $\delta>0$,
	\begin{equation}\label{e:fbdsde}
		\left\{
		\begin{aligned}
			-dy_t=&\big[f^{\alpha_0}(t,\zeta_t,\mu_t)+\delta(f(t,\overline \zeta_t,\overline \mu_t)+k_3\overline p_t)+f_0(t)\big]dt\\&+\big[g^{\alpha_0}(t,\zeta_t,\mu_t)+\delta(g(t,\overline \zeta_t,\overline \mu_t)+k_3\overline q_t)+g_0(t)\big]d\overleftarrow B_t-z_tdW_t,\\
			dp_t=&\big[F^{\alpha_0}(t,\zeta_t,\mu_t)+\delta(F(t,\overline \zeta_t,\overline \mu_t)-k_2\overline y_t)+F_0(t)\big]dt\\&+\big[G^{\alpha_0}(t,\zeta_t,\mu_t)+\delta(G(t,\overline \zeta_t,\overline \mu_t)-k_2\overline z_t)+G_0(t)\big]dW_t-q_td\overleftarrow B_t,\\
			y_T=&\xi,\,\,p_0=\Psi^{\alpha_0}(y_0,\mathcal L(y_0))+\delta(\Psi(\overline y_0,\mathcal L(\overline y_0))-\overline y_0)+\Psi_0.
		\end{aligned}
		\right.
	\end{equation}
In order to prove that \eqref{parameter-e}  with $\alpha=\alpha_0+\delta$ has a solution (for sufficiently small $\delta$),  it suffices to show that the mapping 		    $I_{\alpha_0, \delta}(\overline\zeta)=\zeta$ defined via \eqref{e:fbdsde} is a contraction mapping on $L^2_\mathcal F([0,T];\R^n\times \R^n\times\R^{n\times l}\times\R^{n\times d})$. For this purpose, we will obtain some estimations first.

  Denote
	\begin{align*}
		&\widehat \zeta=(\widehat y,	\widehat p,\widehat z,	\widehat q)=(y- y',p-p',z-z',q-q'),\\&\widehat{\overline \zeta}=(\widehat{\overline y},	\widehat{\overline p},\widehat{\overline z},	\widehat{\overline q})=(\overline y-\overline y',\overline p-\overline p',\overline z-\overline z',\overline q-\overline q').
	\end{align*}
	Applying the product rule \eqref{e:prod-rule} to $\big<\widehat y_t,\widehat p_t\big>$ yields
	\begin{align*}
		&\alpha_0\E\big[\big<\widehat y_0,\Psi(y_0,\mathcal L(y_0))-\Psi(y'_0,\mathcal L(y'_0))\big>\big]+(1-\alpha_0)\E [|\widehat y_0|^2]\\&\hspace{2em}+\delta\E\big[\big<\widehat y_0,-\widehat{\overline y}_0+                                              \Psi(\overline y_0,\mathcal L(\overline y_0))-\Psi(\overline y'_0,\mathcal L(\overline y'_0))\big>\big]\\&= \E\int_0^T\big\{\alpha_0\big<\mathcal A(t,\zeta_t,\mu_t)-\mathcal A(t,\zeta'_t,\mu'_t),\widehat \zeta_t\big>+\delta\big<\mathcal A(t,\overline \zeta_t,\overline \mu_t)-\mathcal A(t,\overline \zeta'_t,\overline \mu'_t),{\widehat \zeta}_t\big>\big\}dt\\&\hspace{2em}-(1-\alpha_0)\E\int_0^T\big\{
		k_3|\widehat p_t|^2+	k_3\|\widehat q_t\|^2+k_2|\widehat y_t|^2+k_2	\|\widehat z_t\|^2
		\big\}dt\\&\hspace{2em}+\delta\E\int_0^T\big\{k_3\big<\widehat {\overline p}_t, \widehat p_t\big>+k_3\big<\widehat {\overline q}_t, \widehat q_t\big>+k_2\big<\widehat {\overline y}_t, \widehat y_t\big>+k_2\big<\widehat{\overline z}_t, \widehat z_t\big>\big\}dt.
	\end{align*}
	By (A1)-(A2), we have 	\begin{align*}
		&(\alpha_0k_4+1-\alpha_0)\E\big[|\widehat y_0|^2\big]+\E\int_0^T\big\{k_3\big(|\widehat p_t|^2+
		\|\widehat q_t\|^2\big)+k_2\big(|\widehat y_t|^2+\|\widehat z_t\|^2\big)\big\}dt \\&\leq \delta\Bigg\{\E\int_0^T\big\{k_3\big(\tfrac12|\widehat {\overline p}_t|^2+\tfrac12|\widehat p_t|^2\big)+k_3\big(\tfrac12\|\widehat {\overline q}_t\|^2+\tfrac12\|\widehat q_t\|^2\big)+k_2\big(\tfrac12|\widehat {\overline y}_t|^2+\tfrac12|\widehat y_t|^2\big)\\&\hspace{4em}+k_2\big(\tfrac12\|\widehat {\overline z}_t\|^2+\tfrac12\|\widehat z_t\|^2\big)+|\widehat \zeta_t||\mathcal A(t,\overline \zeta_t,\overline \mu_t)-\mathcal A(t,\overline \zeta'_t,\overline \mu'_t)|\big\}dt	\\&\hspace{2em}+\E\big[\big(\tfrac{1}{2}|\widehat y_0|^2+\tfrac{1}{2}|\widehat{\overline y}_0|^2\big)+|\Psi(\overline y_0,\mathcal L(\overline y_0))-\Psi(\overline y'_0,\mathcal L(\overline y'_0))||\widehat y_0|\big]\Bigg\}\\&\leq\delta\Bigg\{\E\int_0^T\big\{k_3\big(\tfrac12|\widehat {\overline p}_t|^2+\tfrac12|\widehat p_t|^2\big)+k_3\big(\tfrac12\|\widehat {\overline q}_t\|^2+\tfrac12\|\widehat q_t\|^2\big)+k_2\big(\tfrac12|\widehat {\overline y}_t|^2+\tfrac12|\widehat y_t|^2\big)\\&\hspace{4em}+k_2\big(\tfrac12\|\widehat {\overline z}_t\|^2+\tfrac12\|\widehat z_t\|^2\big)+k_1\big(|\widehat\zeta_t|^2+\tfrac 12|\widehat {\overline \zeta}_t|^2+\tfrac12 W_2^2(\overline\mu_t,\overline \mu'_t)\big)\big\}dt\\&\hspace{2em }+\E\big[\big(\tfrac{1}{2}|\widehat y_0|^2+\tfrac{1}{2}|\widehat{\overline y}_0|^2\big)+k_1\big(|\widehat y_0|^2+\tfrac {1}{2}|\widehat{\overline y}_0|^2+\tfrac{1}{2}W_2^2(\mathcal L(\overline y_0),\mathcal L(\overline y'_0))\big) \big]\Bigg\}.
	\end{align*}
	Noting
		\begin{align*}
		W_2^2(\overline\mu_t,\overline \mu'_t)\leq \E\big[|\widehat {\overline \zeta}_t|^2\big],\hspace{1em}\text{and}\hspace{1em}
		W_2^2(\mathcal L(\overline y_0),\mathcal L(\overline y'_0))\leq \E\big[|\widehat {\overline y}_0|^2\big],
	\end{align*}
	we can find a constant $K_1>0$ depending only on $k_1, k_2, k_3$  such that 
	\begin{align*}
		&(\alpha_0k_4+1-\alpha_0)\E\big[|\widehat y_0|^2\big]+\E\int_0^T\big\{k_2\big(|\widehat y_t|^2+\|\widehat z_t\|^2\big)+k_3\big(|\widehat p_t|^2+\|\widehat q_t\|^2\big)\big\}dt\notag\\&\leq \delta K_1\left\{\E\int_0^T \big(|\widehat \zeta_t|^2+|\widehat{\overline \zeta}_t|^2\big)dt+\E\big[|\widehat  y_0|^2+|\widehat {\overline y}_0|^2\big]\right\}.
	\end{align*}
Noting $(\alpha_0k_4+1-\alpha_0)\geq \min\{1,k_4\}$, we have
	\begin{align}\label{e:est-1}
	&\min\{1, k_4\}\E\big[|\widehat y_0|^2\big]+\E\int_0^T\big\{k_2\big(|\widehat y_t|^2+\|\widehat z_t\|^2\big)+k_3\big(|\widehat p_t|^2+\|\widehat q_t\|^2\big)\big\}dt\notag\\&\leq \delta K_1\left\{\E\int_0^T \big(|\widehat \zeta_t|^2+|\widehat{\overline \zeta}_t|^2\big)dt+\E\big[|\widehat  y_0|^2+|\widehat {\overline y}_0|^2\big]\right\}.
\end{align}
Applying It\^o's formula \eqref{e:ito} to $|\widehat y_t|^2$, we have 
\begin{align*}
	&\E\big[|\widehat y_t|^2\big]+\E\int_t^T\|\widehat z_s\|^2ds\\&=2\E\int_t^T\big\{\big<\widehat y_s,\alpha_0(f(s,\zeta_s,\mu_s)-f(s,\zeta'_s,\mu'_s))-(1-\alpha_0)k_3\widehat p_s\big>\\&\hspace{4em}+\big<\widehat y_s,\delta (f(s,\overline\zeta_s,\overline\mu_s)-f(s,\overline\zeta'_s,\overline\mu'_s))+\delta k_3\widehat{\overline p}_s\big>\big\}ds\\&\quad +\E\int_t^T\big\|\alpha_0(g(s,\zeta_s,\mu_s)-g(s,\zeta'_s,\mu'_s))-(1-\alpha_0)k_3\widehat q_s\\&\hspace{4em}+\delta (g(s,\overline\zeta_s,\overline\mu_s)-g(s,\overline\zeta'_s,\overline\mu'_s))+\delta k_3\widehat{\overline q}_s\big\|^2ds.
\end{align*}
By the Lipschitz conditions (A1) and the Gronwall's inequality,  we can find a constant $K_2$ depending on only $k_1, k_2, k_3, \lambda_1, \lambda_2$ such that
\begin{equation}\label{e:est-2}
	\begin{aligned}
		&\sup_{t\in[0,T]} \E\big[|\widehat y_t|^2\big]  \leq K_2\left\{\delta \E\int_0^T|\widehat {\overline\zeta}_t|^2dt+\E\int_0^T\big\{|\widehat p_t|^2+\|\widehat q_t\|^2\big\}dt\right\},\\
		&	\E\int_0^T\big\{|\widehat y_t|^2+\|\widehat z_t\|^2\}dt \leq K_2 (T\vee 1)\left\{\delta  \E\int_0^T|\widehat {\overline\zeta}_t|^2dt+ \E\int_0^T\big\{|\widehat p_t|^2+\|\widehat q_t\|^2\big\}dt\right\}.
		\end{aligned}
\end{equation}
Similarly, the application of  It\^o's formula \eqref{e:ito} to $|\widehat p_t|^2$ yields
\begin{align*}
	&\E\big[|\widehat p_t|^2\big]+\E\int_0^t\|\widehat q_s\|^2ds\\&=\E\big[\big|\alpha\big\{\Psi(y_0, \cL(y_0))-\Psi(y_0', \cL(y_0'))\big\}+(1-\alpha_0)\widehat y_0\\&\hspace{3em}+\delta\big\{ \Psi(\overline y_0, \cL(\overline y_0))-\Psi(\overline{y}'_0, \cL(\overline{y}'_0 ))-\widehat{\overline y}_0\big\}\big|^2\big]\\
	&\hspace{1em} +2\E\int_0^t\big\{\big<\widehat p_s,\alpha_0(F(s,\zeta_s,\mu_s)-F(s,\zeta'_s,\mu'_s))+(1-\alpha_0)k_2\widehat y_s\big>\\&\hspace{4.5em}+\big<\widehat p_s,\delta (F(s,\overline\zeta_s,\overline\mu_s)-F(s,\overline\zeta'_s,\overline\mu'_s))-\delta k_2\widehat{\overline y}_s\big>\big\}ds\\&\hspace{1em} +\E\int_0^t\big\|\alpha_0(G(s,\zeta_s,\mu_s)-G(s,\zeta'_s,\mu'_s))+(1-\alpha_0)k_2\widehat z_s\\&\hspace{4.5em}+\delta (G(s,\overline\zeta_s,\overline\mu_s)-G(s,\overline\zeta'_s,\overline\mu'_s))-\delta k_2\widehat{\overline z}_s\big\|^2ds
\end{align*}
By the Lipschitz conditions (A1) and the Gronwall's inequality, we can deduce that there exists $K_3$ depending only on $k_1, k_2, k_3,\lambda_1, \lambda_2$ such that 
\begin{equation}\label{e:est-3}
\begin{aligned}
&	\E\int_0^T\big\{|\widehat p_t|^2+\|\widehat q_t\|^2\}dt \leq K_3\left\{\delta \E\left(\int_0^T|\widehat {\overline\zeta}_t|^2dt +|\widehat {\overline y}_0|^2 \right)+ \E\left(\int_0^T\big\{|\widehat y_t|^2+\|\widehat z_t\|^2\big\}dt + |\widehat  y_0|^2\right)\right\}.
\end{aligned}
\end{equation}

Now, in order to obtain the contraction of the mapping $I_{\alpha_0,\delta}$ for small $\delta$, it suffices to prove the following estimation from \eqref{e:est-1}-\eqref{e:est-3},
\begin{align}\label{e:contraction}
	\E\big[|\widehat y_0|^2\big]+\E\Big[\int_0^T\big|\widehat \zeta_t\big|^2dt\Big]\leq \delta K \left(\E\big[|\widehat{\overline y}_0|^2\big]+\E\Big[\int_0^T\big|\widehat {\overline\zeta}_t\big|^2dt\Big]\right),
\end{align}
for some positive constant $K$ depending  on $k_1, k_2, k_3, k_4,  \lambda_1, \lambda_2$ and $T$. The proof is split  in three cases according to the positiveness of $k_2, k_3, k_4$. When $k_2,k_3, k_4>0,$ \eqref{e:contraction} is a direct consequence of  \eqref{e:est-1},  when  $k_2>0,k_3=0, k_4>0$, \eqref{e:contraction} follows from \eqref{e:est-1} and \eqref{e:est-3},  and when  $k_2= 0,k_3>0, k_4\ge 0$,  it follows from \eqref{e:est-1} and \eqref{e:est-2}.
	
	The proof is concluded.
	\end{proof}
	
We are ready to present our main result in this section.
 
	\begin{theorem}\label{ex-uni-fbdsde}
		Under the assumptions (A1)-(A2), there exists a unique solution  $(y,p,z,q)\in L^2_\mathcal F([0,T];\R^n\times\R^n\times \R^{n\times l}\times \R^{n\times d})$ to equation \eqref{mf-bdsde}.
	\end{theorem}
\begin{proof}
The existence and uniqueness of the solution follows  from the  well-posedness of  equation \eqref{alpha=0} by \cite[Proposition 3.6]{10hpw} and Lemma \ref{pertubation}. We  also provide a direct proof for the uniqueness as follows.  

Let $\zeta^1=(y^1,p^1,z^1,q^1)$ and $\zeta^2=(y^2,p^2,z^2,q^2)$ be two solutions of \eqref{mf-bdsde}. Denote $\widehat \zeta=(\widehat y,\widehat p,\widehat z,\widehat q)=(y^1-y^2,p^1-p^2,z^1-z^2,q^1-q^2)$. Applying It\^o's formula to $\la \widehat y_t,\widehat p_t\ra$, yields that
	\begin{align*}
		&\E\left[\la\Psi(y_0^1,\mathcal L(y_0^1))-\Psi(y_0^2,\mathcal L(y_0^2)), \widehat y_0\ra\right]\\&=\E\int_0^T\la\mathcal A(t,\zeta^1_t,\mu_t^1)-\mathcal A(t,\zeta_t^2,\mu_t^2),\widehat \zeta_t\ra dt,
	\end{align*}
	where we recall that $\mathcal A=(-F, f, -G, g)$. 	By the monotonicity condition (A2), we have that 
	\begin{align}\label{e:unique}
		 k_4\E\big[|\widehat y_0|^2\big]\leq-k_2\E\int_0^T\big\{|\widehat y_t|^2+\|\widehat z_t\|^2\big\}dt-k_3\E\int_0^T\big\{\|\widehat p_t|^2+\|\widehat q_t\|^2\big\}dt\leq 0.
	\end{align}
If $k_2, k_3>0$, it yields directly   $\zeta^1=\zeta^2$. If  $k_2$ or $k_3$ is 0, say, $k_2=0$ and  $k_3>0$, we have $p_1=p_2$ and $q_1=q_2$ by \eqref{e:unique}, and the uniqueness of $(y, z)$ follows from the classical result of BDSDEs (see \cite{94pp}). 
The proof is completed.
\end{proof}

\begin{remark}
When the mean-field FBDSDE \eqref{mf-bdsde} is reduced to classical FBDSDE (without mean field), Theorem \ref{ex-uni-fbdsde} recovers the existence and uniqueness result  obtained in  \cite[Theorem 2.2]{03ps}.  Our result is also compatible with \cite[Theorem 2.6]{99pw} when \eqref{mf-bdsde} degenerates to classical FBSDE. 
\end{remark}

\section{Examples}\label{example}
In this section,  we apply our results obtained in preceding sections to  some special cases. For simplicity, we assume $n=l=d=k=1$ throughout this section unless otherwise specified.

\subsection{Scalar interaction}  In this subsection, we consider the scalar interaction type control problem,  in which the dependence upon probability measure is through the moments of the probability measure. 

  More precisely, we assume that the coefficients in the state equation \eqref{state-e} take the following form,
\begin{align*}
	&f(t,y,z,u,\mu)=\hat f\big(t,y,z,u,\int\varphi d\mu\big),\,\,\,\,g(t,y,z,u,\mu)=\hat g\big(t,y,z,u,\int \phi d\mu\big),\\&h(t,y,z,u,\mu)=\hat h\big(t,y,z,u,\int\psi d\mu \big),\,\,\,\,\Phi(y,\mu_y)=\hat\Phi\big(y,\int \gamma d\mu_y\big),
\end{align*}
for functions   $\varphi,\phi,\psi:\R^3\to \R$ and $\gamma:\R\to \R$  with at most quadratic growth, and  functions $\hat f$, $\hat g$, $\hat h: [0,T]\times \R^3\times \R\to \R$, $\hat \Phi: \R\times \R \to \R$ satisfying proper regularity conditions.  Here $\int \varphi d\mu:=\int_{\R^3}\varphi(y,z,u)d\mu(y,z,u)= \E[\varphi(Y,Z, U)]$ where $(Y,Z,U)$ is a random vector with $\cL(Y,Z,U)=\mu$. 

Similar to Example \ref{example-h} in Section \ref{2}, the L-derivatives of $f, g, h$ and $\Phi$ can be calculated via $\hat f, \hat g, \hat h$ and $\hat \Phi$ respectively.  For instance, 
\begin{align*}
	\partial_{\mu_y} f\big(t, y,z, u, \cL(Y, Z, U)\big)(Y,Z,U)=\partial_{r}\hat f \big(t,y,z, u, \E[\varphi(Y,Z,U)]\big) \partial_y \varphi (Y, Z,U),
\end{align*}
where $\partial_r \hat f$ denotes the partial derivative with respect to the term $\E[\varphi(Y,Z,U)]$.

Set $\Theta_t=(y_t,z_t,u_t)$. Then, the adjoint equation \eqref{e:pq} can be written as
\begin{equation}
	\left\{
	\begin{aligned}
		dp_t=&\bigg\{\partial_y\hat f\big(t,\Theta_t,\E\big[\varphi(\Theta_t)\big]\big)p_t+\widetilde\E\big[\widetilde p_t\partial_r\hat f\big(t,\widetilde \Theta_t,\E\big[\varphi(\Theta_t)\big]\big)\partial_y \varphi(\Theta_t)\big]\\&\quad+\partial_y\hat g\big(t,\Theta_t,\E\big[\phi(\Theta_t)\big]\big)q_t+\widetilde\E\big[\widetilde q_t\partial_r\hat g\big(t,\widetilde \Theta_t,\E\big[\phi(\Theta_t)\big]\partial_y \phi(\Theta_t)\big)\big]\\&\quad+\partial_y\hat h\big(t,\Theta_t,\E\big[\psi(\Theta_t)\big]\big)+\widetilde\E\big[\partial_{r}\hat h\big(t,\widetilde\Theta_t ,\E\big[\psi(\Theta_t)\big]\big)\partial_y \psi(\Theta_t)\big]\bigg\}dt\\
		&+\bigg\{\partial_z\hat f\big(t,\Theta_t,\E\big[\varphi(\Theta_t)\big]\big)p_t+\widetilde\E\big[\widetilde p_t\partial_r\hat f\big(t,\widetilde \Theta_t,\E\big[\varphi(\Theta_t)\big]\big)\partial_z \varphi(\Theta_t)\big]\\&\quad+\partial_z\hat g\big(t,\Theta_t,\E\big[\phi(\Theta_t)\big]\big)q_t+\widetilde\E\big[\widetilde q_t\partial_r\hat g\big(t,\widetilde \Theta_t,\E\big[\phi(\Theta_t)\big]\partial_z \phi(\Theta_t)\big)\big]\\&\quad+\partial_z\hat h\big(t,\Theta_t,\E\big[\psi(\Theta_t)\big]\big)+\widetilde\E\big[\partial_{r}\hat h\big(t,\widetilde\Theta_t ,\E\big[\psi(\Theta_t)\big]\big)\partial_z \psi(\Theta_t)\big]\bigg\}dW_t\\&-q_td\overleftarrow B_t,\,\,\,t\in[0,T],\\
		p_0=&\partial_y\hat\Phi\big(y_0,\E[\gamma(y_0)]\big)+\widetilde\E\big[\partial_{r}\hat\Phi\big(\widetilde y_0,\E[\gamma(y_0)]\big)\partial\gamma(y_0)\big].
	\end{aligned}
	\right.
\end{equation}
The stochastic maximum principle \eqref{variational inequality} obtained in Theorem \ref{nmp} becomes
\begin{align*}
	&\bigg\{\partial_u\hat f\big(t,\Theta_t,\E\big[\varphi(\Theta_t)\big]\big)p_t+\partial_u\hat g\big(t,\Theta_t,\E\big[\phi(\Theta_t)\big]\big)q_t+\partial_u\hat h\big(t,\Theta_t,\E\big[\psi(\Theta_t)\big]\big)\\&\quad +\widetilde\E\big[\widetilde p_t\partial_r\hat f\big(t,\widetilde \Theta_t,\E\big[\varphi(\Theta_t)\big]\big)\partial_u\varphi(\Theta_t)\big]+\widetilde\E\big[\widetilde q_t\partial_r\hat g\big(t,\widetilde \Theta_t,\E\big[\phi(\Theta_t)\big]\big)\partial_u\phi(\Theta_t)\big]\\&\quad +\widetilde\E\big[\partial_r\hat h\big(t,\widetilde \Theta_t,\E\big[\psi(\Theta_t)\big]\big)\partial_u\psi(\Theta_t)\big]\big(a-u_t\big)\geq 0,\,\,\, \text{ for all }  a\in U .
\end{align*}
\subsection{First order interaction}
In this example, we consider the case of first order interaction where the dependence of the coefficients on the probability measure is linear in the following sense 
\begin{align*}
	&f(t,y,z,u,\mu)=\int_{\R^3}\hat f(t,y,z,u,y',z',u')d\mu(y',z',u')=\widetilde\E[\hat f(t, y, z, u, \widetilde Y, \widetilde Z,\widetilde U)],\\
	&g(t,y,z,u,\mu)=\int_{\R^3}\hat g(t,y,z,u,y',z',u')d\mu(y',z',u')=\widetilde\E[\hat g(t, y, z, u, \widetilde Y, \widetilde Z,\widetilde  U)],\\
	&h(t,y,z,u,\mu)=\int_{\R^3}\hat h(t,y,z,u,y',z',u')d\mu(y',z',u')=\widetilde\E[\hat h(t, y, z, u,\widetilde Y,\widetilde Z,\widetilde U)],\\
	&\Phi(y,\mu_y )=\int_{\R}\hat\Phi(y,y')d\mu_y(y')=\widetilde\E[\hat \Phi(y, \widetilde Y)],
\end{align*}
for some functions $\hat f$, $\hat g$, $\hat h$ defined on $\R^3\times \R^3$ and $\hat \Phi$ defined $\R\times \R$ with values on $\R$, where $(\widetilde Y,\widetilde Z,\widetilde U)$ is a random vector with the law $\mu$. 

The state equation \eqref{state-e} with first order interaction corresponds to a type of mean-field BDSDE which may arise naturally in economics, finance and game theorem, etc. We refer to \cite{09bdlp} for a study of mean-field BSDEs with first order interaction via a limit approach.

 Actually, when considering the $N$-players game where each individual state is governed by 
\begin{align*}
	dX_t^i=\frac{1}{N}\sum\limits_{j=1}^N \hat b(t,X_t^i,X_t^j,u_t^i)dt+\sigma dW_t,\,\,i=1,\cdots, N,
\end{align*}
$u_t^i$ denotes the strategy of $i$-th player. The equation can be rewritten as
\begin{align*}
	dX_t^i&= b(t,X_t^i,\overline \mu_t^N,u_t^i)dt+\sigma dW_t,
\end{align*}
where $\overline\mu_t^N=\frac{1}{N}\sum\limits_{j=1}^N\delta_{X_t^j}$ and 
$
	b(t,x,\mu,u)=\int_{\R^n}\hat b(t,x,x',u)d\mu(x').
$
Interaction given by functions of the form is called first order or linear. 

Now, we come back to our control problem in the case of first order interaction. From Example~ \ref{example-h}, $f$, $g$, $h$ are linear with respect to $\mu$, and $\Phi$ is linear in $\mu_y$. For $\Phi$, $\partial_{\mu_y}\Phi(y,\mu_y)(y')=\partial_{y'}\hat \Phi(y,y')$ and similarly, $\partial_{\mu_y}f(t,y,z,u,\mu)(y',z',u')=\partial_{y'}\hat f(t,y,z,u,y',z',u')$. The adjoint equation is
\begin{equation}
	\left\{
	\begin{aligned}
		dp_t=&\Big\{\widetilde\E\big[\partial_y\hat f(t,\Theta_t,\widetilde y_t,\widetilde z_t,\widetilde u_t)p_t+\partial_y\hat g(t,\Theta_t,\widetilde y_t,\widetilde z_t,\widetilde u_t)q_t+\partial_y\hat h(t,\Theta_t,\widetilde y_t,\widetilde z_t,\widetilde u_t)\big]\\&\hspace{0.5em}+\widetilde\E\big[\partial_{ y'}\hat f(t,\widetilde\Theta_t, y_t, z_t, u_t)\widetilde p_t+\partial_{y'}\hat g(t,\widetilde\Theta_t, y_t, z_t, u_t)\widetilde q_t+\partial_{y'}\hat h(t,\widetilde\Theta_t, y_t, z_t, u_t)\big]\Big\}dt\\&+\Big\{\widetilde \E\big[\partial_z\hat f(t,\Theta_t,\widetilde y_t,\widetilde z_t,\widetilde u_t)p_t+\partial_z\hat g(t,\Theta_t,\widetilde y_t,\widetilde z_t,\widetilde u_t)q_t+\partial_z\hat h(t,\Theta_t,\widetilde y_t,\widetilde z_t,\widetilde u_t)\big]\\&\hspace{1em}+\widetilde \E\big[\partial_{z'}\hat f(t,\widetilde\Theta_t, y_t, z_t, u_t)\widetilde p_t+\partial_{z'}\hat g(t,\widetilde\Theta_t, y_t, z_t, u_t)\widetilde q_t+\partial_{z'}\hat h(t,\widetilde\Theta_t, y_t, z_t, u_t)\big]\Big\}dW_t\\&-q_td\overleftarrow B_t,\,\,\,t\in[0,T],\\
		p_0=&\widetilde \E\big[\partial_y\hat \Phi(y_0,\widetilde y_0)+\partial_{y'}\hat \Phi(\widetilde y_0, y_0)\big].
	\end{aligned}
	\right.
\end{equation}
Similarly,  it follows from applying stochastic maximum principle in Theorem \ref{nmp}  that for $\forall a\in U$, 
\begin{align*}
	&\bigg\{\widetilde\E\Big[\partial_u\hat f(t,\Theta_t,\widetilde y_t,\widetilde z_t,\widetilde u_t)p_t+\partial_u\hat g(t,\Theta_t,\widetilde y_t,\widetilde z_t,\widetilde u_t)q_t+\partial_u\hat h(t,\Theta_t,\widetilde y_t,\widetilde z_t,\widetilde u_t)\\&\hspace{1em}+\partial_{ u'}\hat f(t,\widetilde\Theta_t, y_t, z_t, u_t)\widetilde p_t+\partial_{u'}\hat g(t,\widetilde \Theta_t, y_t, z_t, u_t)\widetilde q_t+\partial_{ u'}\hat h(t,\widetilde\Theta_t, y_t, z_t, u_t)\Big]\bigg\}\big(a-u_t\big)\geq 0.
\end{align*}
\subsection{LQ problem}\label{lq}
In this subsection, we will apply  the stochastic maximum principle derived in Section \ref{3} to  a kind of mean-field stochastic linear quadratic control problem with scalar interaction.  In such an LQ model, the drift  and the volatility in \eqref{state-e} are of the form
\begin{align*}
	&f(t,y_t,z_t,u_t,\mathcal L(y_t,z_t,u_t))=f_1y_t+f_2z_t+f_3u_t+\overline f_1\E[y_t]+\overline f_2\E[z_t]+\overline f_3\E[u_t],\\
	&g(t,y_t,z_t,u_t,\mathcal L(y_t,z_t,u_t))=g_1y_t+g_2z_t+g_3u_t+\overline g_1\E[y_t]+\overline g_2\E[z_t]+\overline g_3\E[u_t],
\end{align*}
and the cost functional is assumed to be
\begin{align*}
	J(u)=&\frac 12\E\bigg[\int_0^T\big\{h_1y_t^2+h_2z^2_t+h_3u^2_t+\overline h_1(\E[y_t])^2+\overline h_2(\E[z_t])^2+\overline h_3(\E[u_t])^2\big\}dt\\&\hspace{1em}+\Phi y_0^2+\overline \Phi (\E[y_0])^2\bigg],
\end{align*}
where $f_i, \overline f_i$, $g_i, \overline g_i$, $h_i, \overline h_i$ for $i=1,2,3$ and $\Phi, \overline \Phi$ are given constants satisfying $h_1, h_2,\overline h_1, \overline h_2,\Phi, \overline \Phi\geq 0$ and $h_3,\overline h_3>0$, $|g_2|+|\overline g_2|<1$.  In this setting, 
the Hamiltonian $H$ given by \eqref{Hamiltonian} is 
\begin{align}\label{e:H-LQ}
	H(t,y,z,u,\mu,p,q)=&\big\{f_1y+f_2z+f_3u+\overline f_1\E[y]+\overline f_2\E[z]+\overline f_3\E[u]\big\}p\notag\\+&\big\{g_1y+g_2z+g_3u+\overline g_1\E[y]+\overline g_2\E[z]+\overline g_3\E[u]\big\}q\notag\\+&\tfrac 12\big\{h_1y^2+h_2z^2+h_3u^2+\overline h_1(\E[y])^2+\overline h_2(\E[z])^2+\overline h_3(\E[u])^2\big\},
\end{align}
and the adjoint equation \eqref{e:pq}  is 
\begin{equation}
	\left\{
	\begin{aligned}
		dp_t=&\big\{f_1p_t+\overline f_1\E[p_t]+g_1q_t+\overline g_1\E[q_t]+h_1y_t+\overline h_1\E[y_t]\big\}dt\\+&\big\{f_2p_t+\overline f_2\E[p_t]+g_2q_t+\overline g_2\E[q_t]+h_2z_t+\overline h_2\E[z_t]\big\}dW_t\\-&q_td\overleftarrow B_t,\,\,\,t\in[0,T],\\
		p_0=&\Phi y_0+\overline \Phi\E[y_0].
	\end{aligned}
	\right.
\end{equation}

If we further assume that the control domain $U$ is the whole space $\R$,  the stochastic maximum principle \eqref{variational inequality} yields
\begin{align}\label{lq-mp}
	f_3p_t+\overline f_3\E[p_t]	+g_3q_t+\overline g_3\E[q_t]+h_3u_t+\overline h_3\E [u_t]=0.
\end{align}
Taking expectation, we have
\begin{align}\label{Eu}
	\E[u_t]=-\frac{1}{h_3+\overline h_3}\Big\{(f_3+\overline f_3)\E[p_t]+(g_3+\overline g_3)\E[q_t]\Big\}.
\end{align}
Plugging this into 	\eqref{lq-mp}, we obtain that
\begin{align}\label{u}
	u_t=&-\frac{1}{h_3}\Big\{f_3p_t+\tfrac{1}{h_3+\overline h_3}\big(h_3\overline f_3-\overline h_3f_3\big)\E[p_t] +g_3q_t+\tfrac{1}{h_3+\overline h_3}\big(h_3\overline g_3-\overline h_3g_3\big)\E[q_t]\Big\}.
\end{align}
If the following stochastic Hamiltonian system   
\begin{equation}\label{fbdsde}
	\left\{
	\begin{aligned}
		-dy_t=&\Big\{f_1y_t+f_2z_t+f_3u_t +\overline f_1\E[y_t]+\overline f_2\E[z_t]+\overline f_3\E[u_t]\Big\}dt\\ &+\Big\{g_1y_t+g_2z_t+g_3u_t +\overline g_1\E[y_t]+\overline g_2\E[z_t]+\overline g_3\E[u_t]\Big\}d\overleftarrow B_t-z_tdW_t,\\
		dp_t=&\Big\{f_1p_t+\overline f_1\E[p_t]+g_1q_t+\overline g_1\E[q_t]+h_1y_t+\overline h_1\E[y_t]\Big\}dt\\&+\Big\{f_2p_t+\overline f_2\E[p_t]+g_2q_t+\overline g_2\E[q_t]+h_2z_t+\overline h_2\E[z_t]\Big\}dW_t-q_td\overleftarrow B_t,\\
		y_T=&\xi, ~ p_0=\Phi y_0+\overline \Phi\E[y_0].
	\end{aligned}
	\right.
\end{equation}
with $u_t$ being given by \eqref{u} admits a solution, by the verification theorem in Section \ref{sufficiency},  the control process \eqref{u} is indeed the unique  optimal control.
Substituting \eqref{Eu} and \eqref{u} for $\E[u]$ and $u$  respectively leads to a strong coupling between the forward and backward equations in \eqref{fbdsde}, and we cannot apply Theorem \ref{ex-uni-fbdsde} due to the lack of  monotonicity assumed in condition (A2). 
In the rest of this subsection, we shall prove the existence and uniqueness of the solution under some weaker conditions which are satisfied by \eqref{fbdsde} without the terms of $\E[u_t]$.

Consider the following mean-field FBDSDE
\begin{equation}\label{mf-fbdsde2}
	\left\{
	\begin{aligned}
		-dy_t=&f(t,y_t,Cp_t,z_t,Dq_t,\mathcal L(y_t,Cp_t,z_t,Dq_t))dt\\&+g(t,y_t,Cp_t,z_t,Dq_t,\mathcal L(y_t,Cp_t,z_t,Dq_t))d\overleftarrow B_t-z_tdW_t,\\
		dp_t=&F(t,y_t,p_t,z_t,q_t,\mathcal L(y_t,p_t,z_t,q_t))dt\\&+G(t,y_t,p_t,z_t,q_t,\mathcal L(y_t,p_t,z_t,q_t))dW_t-q_td\overleftarrow B_t.\\
		y_T=&\xi, ~p_0=\Psi(y_0,\mathcal L(y_0)),
	\end{aligned}
	\right.
\end{equation}
where $C$ and $D$ are matrices of dimension $n\times n$. For simplicity, here we set $l=d=1$. We introduce below conditions (B1)-(B2)  which are parallel to but weaker than (A1)-(A2) imposed in Section \ref{4}.
As in Section~\ref{4}, we use the   notations $\zeta=(y,p,z,q)$ and $\mathcal A(t,\zeta,\mu)=(-F,\breve f,-G,\breve g)(t,\zeta,\mu)$, where $\breve f(t,\zeta, \mu)=f(t,y, Cp, z, Dq, \mathcal L(y, Cp, z, Dq))$ and similarly for $\breve g$.  
\begin{enumerate}
	\item [(B1)]There exist constants $c_1\geq 0, c_2>0$ such that
	\begin{align*}
		\E\big[\big<\mathcal A(t,\zeta,\mu)-\mathcal A(t,\zeta',\mu'),\widehat\zeta\big>\big]&\leq -c_1\E\big[|\widehat y|^2+|\widehat z|^2\big]-c_2\E\big[|C\widehat p+D\widehat q|^2\big],\\
		\E\big[\big<\Psi(y,\mu_y)-\Psi(y',\mu_y'), \widehat y\big>\big]&\geq 0,
	\end{align*}
for $ \zeta,\zeta'\in \R^{n}\times\R^{n}\times\R^{n}\times\R^{n}, \mu,\mu'\in \mathcal P_2(\R^{n}\times\R^{n}\times\R^{n}\times\R^{n}), \mu_y,\mu_y'\in \mathcal P_2(\R^n)$, $\widehat\zeta=(\widehat y,\widehat p,\widehat z,\widehat q)=(y-y',p-p',z-z',q-q')$.
\medskip
	\item [(B2)]There exists a  constant $c_3>0$ such that
	\begin{align*}
		|\mathcal A(t,\zeta,\mu)-\mathcal A(t,\zeta',\mu')|&\leq c_3\big(|\zeta-\zeta'|+W_2(\mu,\mu')\big),\\
		|\Psi(y,\mu_y)-\Psi(y',\mu'_y)|&\leq c_3\big(|y-y'|+W_2(\mu_y,\mu'_y)\big).
	\end{align*}
	Moreover, we assume that there exist $\lambda_1,\lambda_2>0$ with $\lambda_1+\lambda_2<1$ such that for $\forall t\in[0,T]$,
	\begin{align*}
		&\big|f(t,y,Cp,z,Dq,\mathcal L(y,Cp,z,Dq))-f(t,y',Cp',z',Dq',\mathcal L(y',Cp',z',Dq'))\big|^2\\&\leq  c_3\big(|\widehat y|^2+|\widehat z|^2+|C\widehat p+D\widehat q|^2+\E\big[|\widehat y|^2+|\widehat z|^2+|C\widehat p+D\widehat q|^2\big]\big),\\
		&\big|g(t,y,Cp,z,Dq,\mathcal L(y,Cp,z,Dq))-g(t,y',Cp',z',Dq',\mathcal L(y',Cp',z',Dq'))\big|^2\\&\leq  c_3\big(|\widehat y|^2+|C\widehat p+D\widehat q|^2+\E\big[|\widehat y|^2+|C\widehat p+D\widehat q|^2\big]\big)+\lambda_1|\widehat z|^2+\lambda_2\E\big[|\widehat z|^2\big],\\
		&\big|F(t,y,p,z,q,\mu)-F(t,y',p',z',q',\mu')\big|^2\leq c_3\big(|\widehat\zeta|^2+W_2^2(\mu,\mu')\big),\\
		&\big|G(t,y,p,z,q,\mu)-G(t,y',p',z',q',\mu')\big|^2\\&\leq c_3\big(|\widehat y|^2+|\widehat z|^2+|\widehat p|^2+\E\big[|\widehat y|^2+|\widehat z|^2+|\widehat p|^2\big]\big)+\lambda_1|\widehat q|^2+\lambda_2\E\big[|\widehat q|^2\big].
	\end{align*}
\end{enumerate}
\begin{theorem}\label{lq-eu}
	Under conditions (B1)-(B2), equation \eqref{mf-fbdsde2} admits a unique solution.
\end{theorem}
\begin{proof}
First we prove the uniqueness. Let $\zeta=(y,z,p,q)$ and $\zeta'=(y',p',z',q')$ be two solutions of \eqref{mf-fbdsde2}. Applying It\^o's formula to $\big<\widehat y_t,\widehat p_t\big>$ and taking expectation yield
\begin{align*}
	\E\big[\big<\widehat y_0,\Psi(y_0,\mathcal L(y_0))-\Psi(y'_0,\mathcal L(y'_0))\big>\big]=\E\int_0^T\big<\mathcal A(t,\zeta_t,\mu_t)-\mathcal A(t,\zeta_t',\mu_t'),\widehat \zeta_t\big>dt. 
\end{align*}
This together with condition (B1) implies
\begin{align*}
	c_2\E\int_0^T\big|C\widehat p_t+D\widehat q_t\big|^2dt\leq 0,
\end{align*}
and recalling that $c_2>0$, we have 
\begin{equation}\label{e:pq0}
|C\widehat p_t+D\widehat q_t|^2=0, \text{ for almost all } t\in[0,T].
\end{equation}
Now we deal with $|\widehat y_t|^2$ in a similar way. Using Lipschitz conditions  on $f$ and $g$ in (B2) and taking \eqref{e:pq0} into account,  we can get
\[\E[|\widehat y_t|^2] +\frac12 \E \int_t^T |\widehat z_s|^2 ds \le c_0  \E\int_t^T |\widehat y_s|^2 ds,\] 
for some positive constant $c_0$. This implies $\widehat y\equiv 0$ by Gronwall's inequality and hence $\widehat z\equiv 0$. The uniqueness of $(p,q)$ then follows directly from classical result for BDSDEs.

To obtain the existence of the solution, we consider the following equation:
\begin{equation}\label{mffbdsdep}
	\left\{
	\begin{aligned}
		-dy_t=&\Big\{\alpha f(t,y_t,Cp_t,z_t,Dq_t,\mathcal L(y_t,Cp_t,z_t,Dq_t))-(1-\alpha)(C^\text{T}Cp_t+C^\text{T}Dq_t)+f_0(t)\Big\}dt\\&+\Big\{\alpha g(t,y_t,Cp_t,z_t,Dq_t,\mathcal L(y_t,Cp_t,z_t,Dq_t))-(1-\alpha)(D^\text{T}Cp_t+D^\text{T}Dq_t)+g_0(t)\Big\}d\overleftarrow B_t\\&-z_tdW_t,\\
		dp_t=&\Big\{\alpha F(t,y_t,p_t,z_t,q_t,\mathcal L(y_t,p_t,z_t,q_t))+F_0(t)\Big\}dt\\&+\Big\{\alpha G(t,y_t,p_t,z_t,q_t,\mathcal L(y_t,p_t,z_t,q_t))+G_0(t)\Big\}dW_t-q_td\overleftarrow B_t,\\
		y_T=&\xi, ~ p_0=\alpha\Psi(y_0,\mathcal L(y_0))+\Psi_0.
	\end{aligned}
	\right.
\end{equation}
Clearly, \eqref{mffbdsdep} with $\alpha=1$ coincides with \eqref{mf-fbdsde2}, and  when $\alpha=0$, the existence and uniqueness follows directly from \cite{94pp}. As in Section \ref{4}, we shall take the method of continuation and prove the result of Lemma \ref{pertubation} under conditions (B1)-(B2). More precisely, given $\alpha_0\in[0,1)$ and $\overline\zeta=(\overline y,\overline p,\overline z,\overline q)\in L^2_\mathcal F([0,T];\R^n\times \R^n\times\R^{n}\times \R^{n})$,  we consider 
\begin{equation}\label{e:linear-fbdsde}
	\left\{
	\begin{aligned}
		-dy_t=&\Big\{\alpha_0 f(t,y_t,Cp_t,z_t,Dq_t,\mathcal L(y_t,Cp_t,z_t,Dq_t))-(1-\alpha_0)(C^\text{T}Cp_t+C^\text{T}Dq_t)\\&\hspace{0.5em}+\delta f(t,\overline y_t,C\overline p_t,\overline z_t,D\overline q_t,\mathcal L(\overline y_t,C\overline p_t,\overline z_t,D\overline q_t))+\delta(C^\text{T}C\overline p_t+C^\text{T}D\overline q_t)+f_0(t)\Big\}dt\\&+\Big\{\alpha_0 g(t,y_t,Cp_t,z_t,Dq_t,\mathcal L(y_t,Cp_t,z_t,Dq_t))-(1-\alpha_0)(D^\text{T}Cp_t+D^\text{T}Dq_t)\\&\hspace{1em}+\delta g(t,\overline y_t,C\overline p_t,\overline z_t,D\overline q_t,\mathcal L(\overline y_t,C\overline p_t,\overline z_t,D\overline q_t))+\delta(D^\text{T}C\overline p_t+D^\text{T}D\overline q_t)+g_0(t)\Big\}d\overleftarrow B_t\\&-z_tdW_t,\\
		dp_t=&\Big\{\alpha_0 F(t,y_t,p_t,z_t,q_t,\mathcal L(y_t,p_t,z_t,q_t))+\delta F(t,\overline y_t,\overline p_t,\overline z_t,\overline q_t,\mathcal L(\overline y_t,\overline p_t,\overline z_t,\overline q_t))+F_0(t)\Big\}dt\\&+\Big\{\alpha_0 G(t,y_t,p_t,z_t,q_t,\mathcal L(y_t,p_t,z_t,q_t))+\delta G(t,\overline y_t,\overline p_t,\overline z_t,\overline q_t,\mathcal L(\overline y_t,\overline p_t,\overline z_t,\overline q_t))+G_0(t)\Big\}dW_t\\&-q_td\overleftarrow B_t,\\
		y_T=&\xi, ~ p_0=\alpha_0\Psi(y_0,\mathcal L(y_0))+\delta\Psi(\overline y_0,\mathcal L(\overline y_0))+\Psi_0,
	\end{aligned}
	\right.
\end{equation}
and shall prove that the mapping $\mathbb I _{\alpha_0,\delta}(\overline \zeta)=(\zeta)$ defined by \eqref{e:linear-fbdsde} is contractive for $\delta$ which is small but independent of $\alpha_0$.

Applying the product rule \eqref{e:prod-rule} to $\big<\widehat y_t, \widehat p_t\big>$, taking expectation and using  (B1)-(B2), 
we can get the following estimation which is parallel to \eqref{e:est-1}: there exists a constant $C_1$ only depending on $c_1,c_2,c_3$ such that 
\begin{align*}
	&\E\int_0^T|C\widehat p_t+D\widehat q_t|^2dt\leq \delta C_1\Bigg\{ \E\int_0^T\big\{|\widehat {\overline \zeta}_t|^2+|\widehat \zeta_t|^2\big\}dt+\E\big[|\widehat {\overline y}_0|^2+|\widehat y_0|^2\big]\Bigg\}.
\end{align*}
Applying It\^o's formula to $|\widehat y_t|^2$ and taking expectation, we can get the following estimates:
\begin{align*}
	&\E\big[|\widehat y_t|^2\big]\leq C_2\E\int_0^T|C\widehat p_t+D\widehat q_t|^2dt+\delta C_2\E\int_0^T|\widehat {\overline \zeta}_t|^2dt,\\
	&\E\int_0^T\{ |\widehat y_t|^2+|\widehat z_t|^2\}dt\leq C_3\E\int_0^T|C\widehat p_t+D\widehat q_t|^2dt+\delta C_3\E\int_0^T|\widehat {\overline \zeta}_t|^2dt.
\end{align*}
Similarly,  we can also get
\begin{align*}
	&\E\int_0^T\{|\widehat p_t|^2+|\widehat q_t|^2\}dt\leq C_4\E\int_0^T\{|\widehat y_t|^2+|\widehat z_t|^2\}dt+\delta C_4\E\int_0^T|\widehat {\overline \zeta}_t|^2dt+C_4\E\big[|\widehat y_0|^2\big]+\delta C_4\E\big[|\widehat{\overline y} _0|^2\big].
\end{align*}
Combining the above estimates, we can find a constant $L$ only dependent on $c_1,c_2,c_3,\lambda_1,\lambda_2$ and $T$, such that
\begin{align*}
	\E\int_0^T|\widehat \zeta_t|^2dt+\E\big[|\widehat y_0|^2\big]\leq \delta L\left(\E\int_0^T|\widehat {\overline \zeta}_t|^2dt+\E\big[|\widehat {\overline y}_0|^2\big]\right).
\end{align*}
Hence, if we choose $\delta=\tfrac{1}{2L}$,   $\mathbb I_{\alpha_0, \delta}$ is a contraction mapping and thus equation \eqref{mffbdsdep} admits a solution for $\alpha=\alpha_0+\delta$.  Noting that the choice of $\delta$ is independent of $\alpha_0$, one can repeat this procedure and show that \eqref{mffbdsdep} has a solution for all $\alpha\in[0,1]$. In particular, this implies the existence of solution to \eqref{mf-fbdsde2}. 

The proof is concluded. 
\end{proof}

Now, we reconsider the stochastic linear quadratic problem that does not depend on the distribution of the control process, i.e.,  $\overline f_3,\overline g_3, \overline h_3=0$. 
In such situation, the optimal control $u$ given by \eqref{u} becomes
\begin{align}\label{u2}
	u_t=-\tfrac{1}{h_3}(f_3p_t+g_3q_t),
\end{align}
and the Hamiltonian system \eqref{fbdsde} now is 
\begin{equation}\label{fbdsde-22}
	\left\{
	\begin{aligned}
		-dy_t=&\Big\{f_1y_t+f_2z_t-\tfrac{f_3}{h_3}(f_3p_t+g_3q_t) +\overline f_1\E[y_t]+\overline f_2\E[z_t]\Big\}dt\\ &+\Big\{g_1y_t+g_2z_t-\tfrac{g_3}{h_3}(f_3p_t+g_3q_t) +\overline g_1\E[y_t]+\overline g_2\E[z_t]\Big\}d\overleftarrow B_t-z _tdW_t,\\
		dp_t=&\Big\{f_1p_t+\overline f_1\E[p_t]+g_1q_t+\overline g_1\E[q_t]+h_1y_t+\overline h_1\E[y_t]\Big\}dt\\&+\Big\{f_2p_t+\overline f_2\E[p_t]+g_2q_t+\overline g_2\E[q_t]+h_2z_t+\overline h_2\E[z_t]\Big\}dW_t-q_td\overleftarrow B_t,\\
		y_T=&\xi,  p_0=\Phi y_0+\overline \Phi\E[y_0].
	\end{aligned}
	\right.
\end{equation}
It can be easily checked that the coefficients in \eqref{fbdsde-22} satisfy (B1)-(B2) (we remark that the monotonicity condition in (A2) is not satisfied, though). By Theorem \ref{lq-eu}, there exists a unique  solution to \eqref{fbdsde-22}. Thus, equations \eqref{u2} together with  \eqref{fbdsde-22}  provides a unique optimal control for the mean-field backward doubly stochastic LQ problem without involving the distribution of control. 

\begin{remark}
When the mean-field FBDSDE \eqref{mf-fbdsde2} is reduced to classical FBDSDE (without mean field), Theorem \ref{lq-eu} recovers the existence and uniqueness result  obtained in  \cite[Theorem 3.8]{10hpw}.
\end{remark}

\section*{Acknowledgements}
The authors would like to thank Tianyang Nie for his helpful discussions. J. Song is partially supported by Shandong University (Grant No. 11140089963041) and the National Natural Science Foundation of China (Grant No. 12071256).

\bibliographystyle{plain}
\bibliography{Reference-mfb}
\end{document}